\documentclass[11pt]{amsart}
\usepackage{geometry}                
\usepackage{graphicx}
\usepackage{amssymb}
\usepackage{amscd}
\usepackage{epstopdf}
\title{Integral canonical models for automorphic vector bundles of Abelian Type}
\author{Tom Lovering}
\newtheorem{thm}[subsubsection]{Theorem}
\newtheorem*{thm*}{Theorem}
\newtheorem{lem}[subsubsection]{Lemma}
\newtheorem{prop}[subsubsection]{Proposition}

\newcommand{\A}{\mathbb{A}}
\newcommand{\Af}{\mathbb{A}^\infty}
\newcommand{\Afp}{\mathbb{A}^{\infty,p}}
\newcommand{\C}{\mathbb{C}}

\newcommand{\D}{\mathbb{D}}

\newcommand{\G}{\mathbb{G}}

\newcommand{\Gm}{\mathbb{G}_m}

\newcommand{\Q}{\mathbb{Q}}
\newcommand{\Qbar}{\bar{\mathbb{Q}}}
\newcommand{\Qp}{\mathbb{Q}_p}

\newcommand{\R}{\mathbb{R}}
\newcommand{\bS}{\mathbb{S}}
\newcommand{\Z}{\mathbb{Z}}
\newcommand{\Zp}{{\mathbb{Z}_p}}

\newcommand{\cA}{\mathcal{A}}
\newcommand{\cB}{\mathcal{B}}

\newcommand{\cF}{\mathcal{F}}
\newcommand{\cG}{\mathcal{G}}
\newcommand{\cH}{\mathcal{H}}

\newcommand{\cJ}{\mathcal{J}}

\newcommand{\cL}{\mathcal{L}}

\newcommand{\cO}{\mathcal{O}}
\newcommand{\cP}{\mathcal{P}}

\newcommand{\cS}{\mathcal{S}}

\newcommand{\cV}{\mathcal{V}}

\newcommand{\cX}{\mathcal{X}}

\newcommand{\fm}{\mathfrak{m}}

\newcommand{\fM}{\mathfrak{M}}

\newcommand{\fS}{\mathfrak{S}}
\newcommand{\fX}{\mathfrak{X}}
\newcommand{\fY}{\mathfrak{Y}}

\newcommand{\fp}{\mathfrak{p}}

\newcommand{\ad}{\operatorname{ad}}

\newcommand{\Aut}{\operatorname{Aut}}
\newcommand{\Spec}{\operatorname{Spec}}

\newcommand{\Hom}{\operatorname{Hom}}
\newcommand{\Ker}{\operatorname{Ker}}

\newcommand{\Fil}{\operatorname{Fil}}
\newcommand{\gr}{\operatorname{gr}}
\newcommand{\Char}{\operatorname{char}}
\newcommand{\id}{\operatorname{id}}
\newcommand{\Frac}{\operatorname{Frac}}
\newcommand{\Isom}{\operatorname{Isom}}

\newcommand{\uIsom}{\underline{\Isom}}
\newcommand{\Rep}{\operatorname{Rep}}
\newcommand{\vVec}{\operatorname{Vec}}
\newcommand{\Bun}{\operatorname{Vec}}
\newcommand{\Lisse}{\operatorname{Lisse}}
\newcommand{\Mod}{\operatorname{Mod}}
\newcommand{\Par}{\operatorname{Par}}

\newcommand{\lcm}{\operatorname{lcm}}
\newcommand{\nocontentsline}[3]{}
\newcommand{\Sh}{\operatorname{Sh}}
\newcommand{\Gal}{\operatorname{Gal}}
\newcommand{\Res}{\operatorname{Res}}
\newcommand{\tocless}[2]{\bgroup\let\addcontentsline=\nocontentsline#1{#2}\egroup}

\newcommand{\rightiso}{\stackrel{\cong}{\rightarrow}}
\newcommand{\leftiso}{\stackrel{\cong}{\leftarrow}}
\newcommand{\map}[1]{\stackrel{#1}{\rightarrow}}

\usepackage{hyperref}
\hypersetup{
    linktoc=all,
    bookmarksnumbered
}

\begin{document}
\maketitle

\begin{abstract}
We define and construct integral canonical models for automorphic vector bundles over Shimura varieties of abelian type. 

More precisely, we first build on Kisin's work to construct integral canonical models over $\cO_E[1/N]$ for Shimura varieties of abelian type with hyperspecial level at all primes not dividing $N$ compatible with Kisin's construction. We then define a notion of an integral canonical model for the standard principal bundles lying over Shimura varieties and proceed to construct them in the abelian type case. With these in hand, one immediately also gets integral models for automorphic vector bundles.
\end{abstract}

\tableofcontents
\section{Introduction}
Since the introduction of the abstract theory of Shimura varieties and their canonical models by Deligne \cite{del1}, \cite{del2} following Shimura, and given its promise as a rich generalisation of the classical theory of modular curves, substantial bodies of literature have arisen whose aim is to extend features of the classical theory to this wider context. 

One such feature is the existence of smooth integral models at primes not dividing the level, and their subsequent utility for studying the action of Frobenius on the Galois representations arising from Shimura varieties crucial for the Langlands programme. Such results have been available in the PEL type case for a long time, thanks largely to the programme of Kottwitz, but have recently been extended to the much more general case of abelian type Shimura varieties in work culminating with the recent papers of Kisin \cite{kis2},\cite{kis3}.

Another feature is the manifestation of certain automorphic forms as algebraic sections of a vector bundle over a Shimura variety,  generalising the classical algebraic description of modular forms. The vector bundles playing the role analogous to the tensor powers of the Hodge bundle in the theory of modular forms are the automorphic vector bundles, and canonical models were defined and shown to exist in some cases by Harris \cite{har1} and more generally by Milne \cite{milne1},\cite{milne3}.

In the present paper we start to draw these two threads of the literature together, working in the case of general\footnote{We do require a small technical restriction: that $Z(G)^\circ$ is split by a CM field, but feel it should be possible to remove this restriction, and that it ought to be harmless for most applications.} abelian type Shimura varieties, first filling a gap in the existing literature and showing that Kisin's good integral models can be spread out to smooth models over $\cO_E[1/N]$, then defining a notion of integral canonical models for automorphic vector bundles with a uniqueness property, and finally proving existence in the abelian type case.

More precisely, suppose $(G,\fX)$ is a Shimura datum, with reflex field $E=E(G,\fX)$, $N>1$ and suppose $G$ admits a reductive\footnote{Recall that a (connected) reductive group scheme $G \rightarrow S$ is a smooth affine group scheme with connected reductive geometric fibres.} model $\cG/\Z[1/N]$. Take an open compact $K = K^NK_N \subset G(\Af)$ where $K^N = \prod_{p \nmid  N}\cG(\Zp)$ and $K_N \subset \prod_{p|N} G(\Qp)$ is open compact. Consider the tower
$$\Sh_{K^N}(G,\fX) := \varprojlim_{K_N} \Sh_{K^NK_N}(G,\fX)$$
of quasi-projective $E$-schemes, which comes equipped with an algebraic $\prod_{p|N}G(\Qp)$ action (in fact it carries the action of a slightly larger group as in \cite{del2}). The main result of Kisin's first integral models paper \cite{kis2} tells us that when $(G,\fX)$ is of abelian type this tower admits, for every $v \nmid N$ a smooth integral model $\cS^{\operatorname{Kis},v}_{K^N} /\cO_{E,v}$ to which the $\prod_{p|N}G(\Qp)$-action extends. Note that by a `smooth' model, we mean one which is smooth quasiprojective at any finite level and for which the maps between different finite levels are finite \'etale.

It is natural to ask whether these all come from a global model over $\cO_E[1/N]$. Moreover, Kisin's models also enjoy an `extension property' which characterises them uniquely, so it is natural to ask if furthermore we can find a global model having this extension property. Our first theorem answers this in the affirmative.

\begin{thm*}
With the above setup, suppose $(G,\fX)$ is of abelian type. Then the tower $\Sh_{K^N}(G,\fX)$ admits a smooth integral model $\cS_{K^N}(G,\fX)/\cO_E[1/N]$ to which the $\prod_{p|N}G(\Qp)$-action extends and having the extension property.

Moreover, for any $v \nmid N$ a place of $E$, the model $\cS_{K^N}(G,\fX) \otimes_{\cO_E[1/N]} \cO_{E,v}$ is canonically identified\footnote{This identification is as an $\cO_{E,v}$-scheme with $\prod_{p|N}G(\Qp)$-action and equivariant identification $\iota$ of its generic fibre with $\Sh_{K^N}(G,\fX)$.} with the model $\cS^{\operatorname{Kis},v}_{K^N}$ obtained from Kisin's theory.
\end{thm*}

We call these models `integral canonical models' because the extension property guarantees their uniqueness. 

We make some remarks about the proof. The obvious direct `patching' argument to obtain the result formally from Kisin's models fails because one cannot automatically get the extension property for the models obtained by spreading out, so instead we give a direct construction. The proof in the Hodge type case very closely follows that of \cite{kis2}. In the abelian type case we need several new ideas. 

Firstly, we have no guarantee that $\prod_{p|N} G(\Qp)$ acts transitively on the components of $\Sh_{K^N}(G,\fX)$ so we replace the group theoretic `Deligne-induction' argument by an argument that uses the extension property to reduce the problem to constructing models for each component individually over the ring of integers of the maximal abelian unramified extension of $E$. Secondly, to prove the analogue of \cite[3.4.6]{kis2} without the lemma which follows it, which may not be true in our context, we show that Kisin's `twisting abelian varieties' construction can be carried out using torsors for the centre of $G^{der}$ rather than $G$, and since such torsors are finite, they are simpler to work with. Finally, we believe in fact that unless $G^{der}=G^{ad}$, the group $\Delta(G,G^{ad})$ loc. cit. is not finite, which is necessary to descend the extension property. Fortunately we are able to prove that the corresponding group at level $K^N$ (our situation) is finite, which perhaps also helps to fill a small gap in the original argument.

We then turn to automorphic vector bundles, and following \cite[III]{milne3}, the heart of the matter is really to define and construct integral canonical models for the standard principal bundles $P_{K^N}(G,\fX) \rightarrow \Sh_{K^N}(G,\fX)$. Such bundles are torsors for the quotient 
$$G^c = G/Z_{nc}$$
 by the maximal subtorus $Z_{nc} \subset Z(G)^\circ$ which is $\R$-split but $\Q$-anisotropic. They come equipped with a flat connection $\nabla$, an equivariant $\prod_{p|N}G(\Qp)$-action and a `filtration' which can be written down by giving a map 
$$\gamma: P_{K^N}(G,\fX) \rightarrow Gr_{\mu}$$
where $Gr_{\mu}$ is the flag variety corresponding to a Hodge cocharacter $\mu:\Gm \rightarrow G$ coming from $\fX$. Informally, it is helpful to think of the fibre functors $\omega$ attached to these bundles as giving one the sheaf of de Rham cohomology of a family of motives over $\Sh_{K^N}$ with its Gauss-Manin connection and Hodge filtration.

Any smooth integral model for $P_{K^N}(G,\fX)$ that is a $\cG^c$-torsor would give, for each representation $\rho: \cG^c \rightarrow GL(V_{\Z[1/N]})$, a $\Z[1/N]$ lattice inside $\omega(V \otimes \Q)$. We define an integral canonical model to be one where at `crystalline points' this lattice coincides with a different lattice constructed using Kisin's theory of $\fS$-modules \cite{kis1}. Roughly speaking, working on the generic fibre one has a Galois cover $\Sh_{K^{Np}}(G,\fX) \rightarrow \Sh_{K^N}(G,\fX)$ with Galois group $\cG^c(\Zp)$ so we obtain attached to $\rho$ a lisse $\Zp$-sheaf
$$\cL := \Sh_{K^{Np}}(G,\fX) \times V_{\Zp} / \cG^c(\Zp)$$
on $\Sh_{K^N}(G,\fX)$. Restricting this to a crystalline point $s$, the theory of $\fS$-modules gives a lattice $\D(s^*\cL) \subset D_{dR}(s^*\cL[1/p])$. Thus we have defined lattices (at each prime) inside $\omega(V \otimes \Q)$, and we say a model 
$$(\cP_{K^N}(G,\fX), \iota:\cP_{K^N}(G,\fX) \otimes_{\cO_E[1/N]} E \rightiso P_{K^N}(G,\fX))$$
is canonical if the lattices it generates agree with these lattices coming from $p$-adic Hodge theory. We then check that if such a model exists and there are enough crystalline points (which we check in the abelian type case), it is unique up to canonical isomorphism. We also reserve the term `canonical' for models for which the connection, Hecke action and filtration extend, but these seem to be automatic properties of the models satisfying the lattice condition in the abelian type case (and we would assume in general).

Of course, with this definition, our main theorem is the following.

\begin{thm*}
Let $(G,\fX)$ be a Shimura datum of abelian type and $\cG/\Z[1/N]$ a reductive model for $G$. Then $P_{K^N}(G,\fX)$ has an integral canonical model $(\cP_{K^N},\iota)$.
\end{thm*}

We give a quick summary of the proof. For the `special type' case where $G$ is a torus, we use the theory of CM motives to find lattices in the de Rham cohomology of abelian varieties. For the Hodge type case, the universal abelian variety has an integral model so we take as our starting point the sheaf $\cV$ of its relative de Rham cohomology, and note that the Hodge tensors of \cite[2.2]{kis2} $s_{\alpha,dR} \in \cV_E^\otimes$ extend to $\cV^\otimes$, at which point they may be used to define a functor
$$\cP_{K^N} := \uIsom_{s_\alpha}(V, \cV)$$
which we show is a $\cG$-torsor using several ingredients from \cite{kis2}, and is the required integral canonical model.

For passing from the Hodge type to abelian type case, we need a new idea which may be more widely applicable. Suppose $(G_2,\fX_2)$ is an abelian type Shimura datum of interest. If we let $(G,\fX)$ be a Hodge type datum such that there is an isogeny $G^{der} \rightarrow G_2^{der}$ witnessing that $(G_2,\fX_2)$ is of abelian type, we do not have a map relating $G_2$ to $G$ but only between the derived groups. However, the torsors $P_{K^N}(G,\fX)$ cannot be reduced to $G^{der}$ without passing to $\C$, which loses the information we are interested in. 

Our solution, inspired by Deligne \cite[2.5]{del2}, is to define for each connected Shimura datum $(G^{der},\fX^+)$ and field $E \supset E(G^{der},\fX^+) := E(G^{ad},\fX^{ad})$ a new Shimura datum $(\cB,\fX_{\cB})$ with the property that any Shimura datum $(G,\fX)$ whose reflex field is contained in $E$ and whose connected Shimura datum is $(G^{der},\fX^+)$ admits a canonical map
$$(\cB,\fX_{\cB}) \rightarrow (G,\fX).$$

We then pass from Hodge type to abelian type by first giving a direct construction of a canonical model for $P_{K_{\cB}^N}(\cB,\fX_{\cB})$ given one for $P_{K^N}(G,\fX)$. With this, we are in business, because if we let $(\cB_2,\fX_{\cB,2})$ be the corresponding pair for $(G_2,\fX_2)$ there is a map
$$(\cB,\fX_{\cB}) \rightarrow (\cB_2,\fX_{\cB,2}) \rightarrow (G_2,\fX_2)$$
and we are able to descend the torsor through the map on connected components while pushing out the $\cB^c$-action to a $G_2^c$-action. Finally we translate our results into results for automorphic vector bundles.

While we do not give applications in this paper, we anticipate this construction playing a useful role in several places. It is already used to study integrality of periods in a preprint of Ichino-Prasanna \cite{ip}, and we expect it should be useful in much more general contexts along these lines.

Working at a formal completion at a single place the same construction gives families of strongly divisible filtered $F$-crystals. This theory is developed and used in \cite{l} to establish a new result that the Galois representations formed by taking the cohomology of the Shimura variety with coefficients in the usual lisse sheaves are crystalline, and in appropriate situations Fontaine Laffaille. In particular this gives a new proof of part of local global compatibility at $l=p$ in certain cases, as well as providing a new tool for studying the $p$-adic geometry of Shimura varieties and $p$-adic automorphic forms.

\section*{Acknowledgements}
The author would like to extend the greatest of thanks to his PhD supervisor Mark Kisin, under whose guidance this project was completed. This project owes much to his sharp insights, but also his optimism, energy and patience. 

We would also like to thank Jack Thorne, after a conversation with him in 2013 provided the initial motivation for this project. For other useful conversations we thank Chris Blake, George Boxer, Lukas Brantner, Justin Campbell, Kestutis Cesnavicius, Erick Knight, Ananth Shankar, Jack Shotton, Rong Zhou, and Yihang Zhu.

This paper is part of the author's PhD thesis completed at and funded by Harvard University, and begun while the author was on a Kennedy Scholarship.

\section{Integral canonical models for Shimura varieties of abelian type}
\subsection{Extension property}
\subsubsection{}
We first recall the extension property used to characterise Shimura varieties at infinite level. Let $R$ be a domain with field of fractions $K$. We call $S/R$ a \emph{test scheme} if it is regular and formally smooth over $R$.\footnote{I believe there is still some controversy over the ``correct'' definition of a test scheme but this should do for our purposes.} Suppose we are given a scheme $X/R$. We say \emph{$X$ has the extension property} if for any test scheme $S/R$ any map $S_K \rightarrow X_K$ extends over $R$.

The following uniqueness statement is well known.
\begin{lem}
\label{extension uniqueness}
Suppose $Y/K$ is a scheme. Then if $X/R$ is a model for it over $R$ and is both a test scheme and has the extension property, $X$ is the unique such model up to canonical isomorphism.
\end{lem}
\begin{proof}
Let $X,X'$ be two such models. Then we are (as part of the data of a model) given maps
$$X_K \rightiso Y \rightiso X'_K.$$
Since $X$ is a test scheme, and since $X'$ has the extension property, this isomorphism extends to an isomorphism
$X \rightiso X'$ of $R$-schemes.
\end{proof}

We also record the following useful formal properties.

\begin{lem}
\label{extension localisation}
Let $R'/R$ be an \'etale or ind-\'etale (or formally smooth) extension of domains with fraction field extension $K'/K$. Suppose $X/R$ satisfies the extension property. Then so does $X \otimes_R R'/R'$.

Conversely, if $R'/R$ is also faithfully flat, then if $X \otimes_R R'/R'$ satisfies the extension property, so does $X/R$.
\end{lem}
\begin{proof}
Let $S'/R'$ be a test scheme. Then $S'$ is regular, and it is formally smooth over $R$ since ind-\'etale algebras are formally \'etale. Suppose we are given $S' \otimes_{R'} K' \rightarrow X_{R'}$, and first compose it with the map $X_{R'} \rightarrow X$. By the extension property for $X/R$, $S' \otimes_{R'} K' = S' \otimes_R K \rightarrow X$ extends to a map $S' \rightarrow X$ over $R$. But since we have a diagram
$$\begin{CD}
   S'  @>>>  X\\ 
@VVV        @VVV\\
    \Spec R' @>>>  \Spec R
\end{CD}
$$
this map must factor through $X \otimes_R R'$, as required, giving the extension property for $X \otimes_R R'/R'$.

For the converse, a test scheme $S/R$ gives rise to a test scheme $S_{R'}/R'$ together with a descent datum $\theta_S$ for $R'/R$. Given a map $S_K \rightarrow X_K$ we obtain $S_{K'} \rightarrow X_{K'}$ compatible with the descent data on both sides. By the extension property this map extends to $S_{R'} \rightarrow X_{R'}$, and being a map of descent data this descends to the desired extension $S \rightarrow X$.
\end{proof}

\begin{lem}
\label{extension pullback}
Let $X/R$ have the extension property and $Y \rightarrow X$ be finite or profinite \'etale. Then $Y/R$ has the extension property.
\end{lem}
\begin{proof}
It clearly suffices to do the finite \'etale case (the profinite one then following formally). Let $K=\Frac(R)$ and suppose we have a test scheme $S$ and a map $S_K \rightarrow Y_K$. Then this map composed with $Y_K \rightarrow X_K$ extends to a map $S \rightarrow X$. Pulling this back we get $S \times_X Y \rightarrow Y$. The map $S_K \rightarrow Y_K$ gives a section of $(S \times_X Y)_K \rightarrow S_K$. But $S \times_X Y \rightarrow S$ is finite \'etale so such sections extend uniquely and we get the required $S \rightarrow Y$. 
\end{proof}

\begin{lem}
\label{extension descent}
Let $Y \rightarrow X$ be finite \'etale, and suppose $Y/R$ has the extension property. Then $X/R$ has the extension property.
\end{lem}
\begin{proof}
By the theory of the \'etale fundamental group and (\ref{extension pullback}) we may assume $Y \rightarrow X$ is Galois. The argument then follows from that of \cite[3.21.4]{moon}.
\end{proof}

We remark that the above result fails for pro-(finite \'etale) extensions. Indeed, Shimura varieties at finite level certainly do not generally have the extension property. For example, for the modular curve at finite level this would imply all elliptic curves have good reduction.

\subsection{Main theorem}
Now, let $(G,\fX)$ be a Shimura datum of abelian type with reflex field $E$.

Fix $S$ a finite nonempty set of finite primes containing all those at which $G$ is ramified, set $N = \prod_{p \in S} p$, a reductive integral model $G_{\Z[1/N]}$ of $G$ (which exists by taking an arbitrary integral model, observing that it is reductive at all but finitely many primes and then gluing in models for the remaining primes). We abusively denote this model by $G$, and let $K^N = \prod_{p \not\in S} G(\Zp)$.

Consider the tower $$\Sh_{K^N} = \varprojlim_{K_N \subset \prod_{p \in S} G(\Qp)} \Sh_{K_NK^N}$$
of Shimura varieties over $E$ with infinite level at the primes dividing $N$ but hyperspecial level at all other primes.

Recall that in this context a \emph{smooth integral model} $\cS_{K^N}$ for $\Sh_{K^N}$ over $\cO_E[1/N]$ is an integral model on which $\prod_{p|N} G(\Q_p)$ acts such that whenever $K_N \subset \prod_{p|N} G(\Qp)$ compact open is sufficiently small that $\Sh_{K_NK^N}$ is a scheme, the model $\cS_{K^N}/K_N$ is smooth quasiprojective, and the maps between such finite level schemes induced by Hecke operators are finite \'etale.

In this chapter we prove the following theorem.
\begin{thm}
\label{S2 thm}
The tower $\Sh_{K^N}$ has a smooth integral model $\cS_{K^N}/\cO_E[1/N]$ satisfying the extension property. For any $v \nmid N$ the localisation of this model at $v$ agrees with Kisin's smooth integral model.
\end{thm}

The following together with (\ref{extension uniqueness}) gives us that this model is canonical.
\begin{lem}
If $\cS_{K^N}$ is a smooth integral model in the above sense, then it is regular and formally smooth (in the usual sense).
\end{lem}
\begin{proof}
That it is formally smooth follows formally because it is a limit of smooth schemes with finite \'etale transition maps. That it is regular follows from the same argument as \cite[2.4]{milne2}.
\end{proof}

\subsubsection{}
We now set out to prove the theorem, following the outline of Kisin's strategy, first using the modular interpretation of Siegel varieties to get going, then making constructions in the Hodge and abelian cases close enough to his that the smoothness, extension and comparison properties follow by direct comparison or in a very similar way.

\subsection{Siegel case}
For this case we recall the following theorem \cite[7.9]{mum}.
\begin{thm}[Mumford]
If $n \geq 6^g d \sqrt{g!}$ then the fine moduli scheme $\cA_{g,d,n}$ of abelian schemes of dimension $g$ together with a polarisation of degree $d$ and a level $n$ structure exists, and is quasi-projective over $\Z$. 

Moreover, for $N=\lcm(d,n)$ these moduli schemes are smooth over $\Z[1/N]$.
\end{thm}

\subsubsection{}
We may also (since we may take quotients of quasiprojective schemes by finite free group actions) form moduli $\cA_{g,d,K_N}$ with $GSp_{2g}(\hat{\Z}^N) K_N$-level structures for all $K_N \subset \prod_{p|N} GSp_{2g}(\Qp)$ sufficiently small, and these are also smooth over $\Z[1/N]$ by etale descent.

\subsubsection{}
Moreover, recall that for $K = K_N GSp_{2g}(\hat{\Z}^N) \subset GSp_{2g}(\Af)$, the Shimura variety $\Sh_K(GSp_{2g},S^\pm)$ is defined over $\Q$ and has a moduli interpretation giving an embedding $\Sh_K(GSp_{2g},S^\pm) \hookrightarrow \cA_{g,d,K_N}$. We define its integral model $$\cS_K := \overline{\Sh_K(GSp_{2g},S^\pm)} \subset \cA_{g,d,K_N}.$$
It is well known that such models are smooth and admit an explicit description as moduli schemes. In particular, they carry a universal abelian scheme defined over $\Z[1/N]$. Moreover the transition maps as we vary $K_N$ are finite \'etale, so we obtain the desired smooth integral models
$$\cS_N = \varprojlim_{K_N} \cS_K / \Z[1/N].$$

We are required to check the following. Note that in preparation for the Hodge type case we need to observe the following holds in the slightly more general context where we assume that $K = \prod K_p$ as above except for some finitely many $p \nmid N$, $K_p$ is not necessarily hyperspecial but merely maximal compact. Luckily with this remark made the argument goes through unchanged.
\begin{prop}
The scheme $\cS_N$ satisfies the extension property.
\end{prop}
\begin{proof}
This follows because the argument of Milne \cite[2.10]{milne2} adapts practically unchanged to our situation. For the reader's convenience we sketch the argument. Let $S$ be a test scheme. A map
$$S_\Q \rightarrow \cS_N$$
gives the data of a triple $(A,\lambda,\eta)$ where $A/S_\Q$ is an abelian scheme, $\lambda$ a polarisation (defined up to a constant) and $\eta$ an infinite level structure at the primes $l|N$. Since $N>1$, the set of such primes is nonempty, and we may let $l$ be one of them.

Now, since $S$ is assumed regular, in particular each of its components is integral. Let $S^0$ be such a component, and denote by $\eta$ its generic point. The infinite level structure at $l$ defined over $S_{\Q}$ in particular trivialises the $l$-adic Tate module of $A_\eta$. By the ``generalised Neron criterion'' \cite[2.13]{milne2} we see that this implies $A_{\eta}$ extends over $S^0$. Since it does so for all components $S^0$ of $S$ and $S$ is normal (so these components do not meet), we deduce that we have extended $A$ to some $\cA/S$.

The polarisation $\lambda$ also extends by \cite[2.14]{milne2}, and the level structures (being at primes away from the characteristic of the base) also obviously extend. This suffices to show the extension property.
\end{proof}

\subsection{Hodge type case}
Suppose we have $(G,\fX)$ a Shimura datum of Hodge type, and fix a symplectic embedding
$$i: (G,\fX) \hookrightarrow (GSp(V_\Q,\psi),S^\pm).$$
Let $N$ be the product of all primes where $G$ is ramified. Recall that we are fixing an integral model $G/\Z[1/N]$ for $G$ and taking the hyperspecial level $K^N = \prod_{p \nmid N} G(\Zp)$ away from $N$. We begin with some group theoretic preliminaries.

\begin{lem}
\label{lattice intersect}
Let $V/\Q$ be a finite dimensional vector space. Take $N\geq 1$ and $p \not |N$, and suppose we have a $\Z[1/pN]$-lattice $\Lambda \subset V$ and a $\Zp$-lattice $L \subset V\otimes\Qp$.

Then $L \cap \Lambda$ is a $\Z[1/N]$-lattice in $V$.
\end{lem}
\begin{proof}
We first observe that for any $\Z[1/pN]$-basis $e_1,...,e_n$ of $\Lambda$, $p^m e_i \subset L$ for all $i=1,...,n$ and some $m$ sufficiently large. In particular $\Lambda \cap L$ contains $\Z[1/pN]$-bases for $\Lambda$. Let us take some such basis $e_1,...,e_n$ such that the $p$-adic volume is maximal (or equivalently such that the index of its $\Zp$-span in $L$ is minimal). We claim this basis generates $\Lambda' = \Lambda \cap L$ as a free $\Z[1/N]$-module.

Since it's a $\Z[1/pN]$-basis for $\Lambda$, its span certainly is a free $\Z[1/N]$-module. Suppose it doesn't generate. Then there is some $y \in \Lambda'$ not in the span of the $e_i$. Since the $e_i$ are a $\Z[1/pN]$-basis we may write $y$ uniquely as
$$y = \sum_i \lambda_i e_i$$
with $\lambda_i \in \Z[1/pN]$. After reordering, let us assume $\lambda_1 = p^{-k} \lambda$ with $\lambda \in \Z[1/N], k \geq 1$ is the coefficient with the largest $p$-adic norm amongst the $\lambda_i$.
Then we can take the new basis $y,e_2,...,e_n$ and it visibly has strictly larger $p$-adic volume, contradicting our original choice of basis and hence the existence of $y$.

\end{proof}

\begin{prop}
\label{glue groups}
Let $G/\Q$ be a reductive group unramified away from $N$, let $N|M$ and $T$ be the set of primes dividing $M/N$. If $2 \in T$, assume further that any factors of $G$ of type B have simply connected derived group.

Then given choices of reductive models $G^M/\Z[1/M]$ and $G_p/\Zp$ for $p \in T$  for $G$, we can find a model $\cG/\Z[1/N]$ isomorphic to each of these.

Moreover, if we have a faithful representation $i:G^M \hookrightarrow GL(V_{\Z[1/M]})$ there is a $\Z[1/N]$-lattice $\Lambda' \subset V_{\Z[1/M]}$ such that $i$ extends to $\tilde{i}: \cG \hookrightarrow GL(\Lambda').$
\end{prop}

\begin{proof}
It obviously suffices to consider the case where $M=pN$. Let $i:G^M \hookrightarrow GL(V_{\Z[1/pN]})$ be a faithful representation. By \cite[2.3.1]{kis2}, and the remark of Madapusi Pera \cite[4.3, footnote]{mp} in the $2 \in T$ case, we can find a $\Zp$-lattice $\Lambda \subset V\otimes \Qp$ such that $G_p(\Qp) = G(\Qp) = G^M_{\Qp} \hookrightarrow GL(V_{\Qp})$ is induced from a map $G_p \rightarrow GL(\Lambda)$ over $\Zp$. 

By (\ref{lattice intersect}) we obtain a $\Z[1/N]$-lattice $\Lambda' = \Lambda \cap V_{\Z[1/pN]}.$ Of course we can canonically identify $\Lambda' \otimes \Z[1/pN] \cong V_{\Z[1/pN]}$, in the context of which we take $\cG/\Z[1/N]$ to be the closure of $$\text{Im}(G^M \hookrightarrow GL(V_{\Z[1/pN]}) \hookrightarrow GL(\Lambda')).$$

We claim this $\cG$ does the job. It's evident that $\cG|_{\Z[1/pN]} \cong G^M$. Since we have a canonical isomorphism $\Lambda' \otimes \Zp \cong \Lambda$ and by the other identifications in the construction we also have $\cG|_{\Zp} \cong G_p$. We also need it to be reductive (i.e. smooth affine with connected reductive geometric fibres). Being a closed subgroup of $GL(\Lambda')$ it's visibly affine, its geometric fibres are reductive by what we already know, and smoothness can be checked fpqc locally, whence it also follows by the identifications we have made.

The second part of the proposition is an immediate consequence of our argument.
\end{proof}
We also echo the remarks in \cite[4.3]{mp}, that in the case of $G$ coming from a Hodge type Shimura datum, the condition on factors of type B is always satsified, by Deligne's classification of symplectic representations. We shall also need the following modification of \cite[2.1.2]{kis2}.

\begin{lem}
\label{embedding levels}
Let $i:(G_1,\fX_1) \hookrightarrow (G_2,\fX_2)$ be an embedding of Shimura data with $K_2^N \subset \prod_{p \nmid N}' G_2(\Qp) =: G_2(\A^{\infty,N})$  compact open, and $K_1 = K_1^N K_{1,N}$ a compact open of $G_1(\Af)$ such that $K_1^N = K_2^N \cap G_1(\A^{\infty,N})$. Suppose $G_1$ and $G_2$ both have centre which is compact modulo its split part.\footnote{We expect the argument can be modified slightly as in Deligne to remove this hypothesis.} Then there exists an open compact subgroup $K_{2,N} \subset \prod_{p|N} G_2(\Qp)$ such that $K_2 := K_{2,N}K_2^N \supset K_1$ and the induced map of $E(G_1,\fX_1)$-schemes
$$\Sh_{K_1}(G_1,\fX_1) \rightarrow \Sh_{K_2}(G_2,\fX_2)_{E(G_1,\fX_1)}$$
is a closed embedding. 
\end{lem}

\begin{proof}
By the same argument as \cite[1.15]{del1} it suffices to check
$$\alpha: \Sh_{K_1^N}(\C) \rightarrow \Sh_{K_2^N}(\C)$$ is injective.

On the level of complex points (by the assumption on centres, which removes the technicalities involving units) this map is
$$\alpha: G_1(\Q) \backslash \fX_1 \times G_1(\Af) / K_1^N \rightarrow G_2(\Q) \backslash \fX_2 \times G_2(\Af) / K_2^N.$$
We prove this is injective by first noting that
$$G_1(\Q) \backslash \prod_{p|N} G_1(\Qp) \rightarrow G_2(\Q) \backslash \prod_{p|N} G_2(\Qp)$$
is injective as in \cite[1.15.3]{del1}. Now fix a set of coset representatives of $G_1(\Q)$ in $\prod_{p|N}G_1(\Qp)$ and note that translating by these, the fibres of $G_1(\Q) \backslash \fX_1 \times G_1(\Af) / K_1^N \rightarrow G_1(\Q) \backslash \prod_{p|N} G_1(\Qp)$ may be identified with $\fX_1 \times G_1(\A^{\infty,N})/K_1^N$. But now since $K_1^N = K_2^N \cap G_1(\A^{\infty,N})$, we have that
$$\fX_1 \times G_1(\A^{\infty,N})/K_1^N \rightarrow \fX_2 \times G_2(\A^{\infty,N})/K_2^N,$$
is injective, and the lemma follows.
\end{proof}

\subsubsection{}
\label{hodge group construction}
We now proceed with the construction in the Hodge type case.

Firstly, by finite-presentedness we note that $i$ is in fact defined over $\Z[1/M]$ for some $M$ divisible by $N$. By (\ref{glue groups}) we can find a lattice $V_{\Z[1/N]} \subset V$ and $\cG/\Z[1/N]$ a reductive model such that (forgetting the symplectic pairing) $i$ is obtained from a map $\cG \hookrightarrow GL(V_{\Z[1/N]})$ and $\cG(\Zp) = G(\Zp)$ for all $p \nmid N$, in particular giving $K^N = \prod_{p\nmid N} \cG(\Zp)$. 

Let $K'^N$ be the stabiliser of $V_{\Z[1/N]}$ in $\prod_{p\nmid N} GSp(V\otimes \Qp,\psi)$, noting that $K'^N$ will be maximal compact but need not be hyperspecial at the primes dividing $M/N$. We also fix a $\Z$-lattice $V_\Z \subset V_{\Z[1/N]}$ and note that for all $K_N \subset \prod_{p|N} G(\Qp)$ sufficiently small $K=K_NK^N$ fixes $V_{\hat{\Z}}.$

\subsubsection{}
Applying the lemma (\ref{embedding levels}) in our setting (since $G$ is Hodge type the hypothesis on the centre holds), and letting $E=E(G,\fX)$, we obtain a closed embedding of Shimura varieties
$$\Sh_{K^N}(G,\fX) \hookrightarrow \Sh_{K'^N}(GSp, S^\pm)_E \subset \cS_{N, \cO_E[1/N]}.$$
Letting $\overline{\Sh_{K^N}} \subset \cS_{N,\cO_E[1/N]}$ be the scheme theoretic closure of this map, the extension property for $\cS_N$ implies the extension property for $\overline{\Sh_{K^N}}$.
Now let $\cS_{K^N}$ be the normalisation of $\overline{\Sh_{K^N}}$. Since test schemes are regular and a fortiori normal, the universal property of normalisation implies that $\cS_{K^N}$ also has the extension property.

\subsubsection{}
We need to check smoothness at finite level. It suffices to check smoothness in a formal neighbourhood of any closed point. When the point has characteristic zero it is in the generic fibre and smoothness is guaranteed. When it has characteristic $p$, we observe that at finite level our construction exactly follows that of Kisin \cite[2.3]{kis2} and Kim-Madapusi Pera \cite[3.5]{kim-mp} for the case $p=2$, and so in particular the necessary local rings are smooth. Hence the $\cS_{K^N}$ give the required integral canonical models in the Hodge type case, compatible with Kisin's by construction.

\subsection{Abelian type case}

We begin by making some more observations about the Hodge type setting. As before we fix connected reductive $G/\Z[1/N]$ belonging to a Shimura datum $(G,\fX)$ of Hodge type with reflex field $E$, and consider the tower $$\Sh_{K^N} = \varprojlim_{K_N} \Sh_{K_NK^N}$$ obtained by fixing $K^N = G(\hat{\Z}^N)$ and letting the level at $p|N$ go to infinity.

\begin{lem}
\label{component field}
The connected component $\Sh_{K^N}^+ \subset \Sh_{K^N,\bar{E}}$ is defined over the maximal abelian extension $E_N/E$ unramified away from $N$.
\end{lem}
\begin{proof}
By \cite[2.2.4]{kis2}, and taking a suitable quotient, we see that $\Sh_{K^N}^+$ is defined over the maximal abelian extension $E^p/E$ unramified away from $p$ for all $p\nmid N$. By the identification $\bigcap_{p \nmid N} E^p = E_N$ inside $E^{ab}$, we deduce that $\Sh_{K^N}^+$ is defined over $E_N$.
\end{proof}

\subsubsection{}
Let $q \nmid N$ be a prime, and recall the following groups which are used to construct Kisin's integral models, and the following results from \cite[3.3]{kis2}. We adopt the usual notations where $G^{ad}(\Q)^+$ is the intersection of $G^{ad}(\Q)$ with the connected component $G^{ad}(\R)$, $G(\Q)_+$ the inverse image of $G^{ad}(\Q)^+$ in $G(\Q)$, $G^{ad}(\Z_{(q)})^+ = G^{ad}(\Q)^+ \cap G^{ad}(\Z_{(q)})$ and $G(\Z_{(q)})_+ = G(\Q)_+ \cap G(\Z_{(q)})_+$.

We have
$$\cA_q(G) := G(\A^{\infty,q})/\overline{Z(\Z_{(q)})} *_{G(\Z_{(q)})_+/Z(\Z_{(q)})} G^{ad}(\Z_{(q)})^+$$
which acts on $\Sh_{K_q} = \varprojlim_{K^q} \Sh_{G(\Z_q)K^q}$, and the subgroup

$$\cA_q^\circ(G) := \overline{G(\Z_{(q)})_+}/\overline{Z(\Z_{(q)})} *_{G(\Z_{(q)})_+/Z(\Z_{(q)})} G^{ad}(\Z_{(q)})^+$$
which acts on a connected component $\Sh_{K_q}^+$. It follows from the argument in \cite[3.3.7]{kis2} that this is precisely the subgroup sending $\Sh_{K_q}^+$ into itself.

\subsubsection{}
We make the following new remarks. The construction of twisting abelian varieties by a $Z$-torsor from \cite[\S3]{kis2} can be carried out by twisting by a $Z^{der} = Z(G^{der})$-torsor instead. Indeed, given $\gamma \in G^{ad}(\Q)^+$ we may take $\cP$ to be the fibre of $\gamma$ along $G^{der} \rightarrow G^{ad}$. It will suffice to check the following.

\begin{prop}
\begin{enumerate}
\item With notation as above, if $\cP'$ is the fibre of $\gamma$ along $G \rightarrow G^{ad}$ then $\cP' = \cP \times^{Z^{der}} Z.$
\item If $V$ is a $\Q$-vector space with an $\cO_Z$-comodule action, $\cP,\cP'$ as above, there is a natural isomorphism
$$(V \otimes_\Q \cO_\cP)^{Z^{der}} \rightiso (V \otimes_\Q \cO_{\cP'})^Z.$$
\end{enumerate}
\end{prop}
\begin{proof}
For (1) we can define (using $\cP' \subset G$) a map $$\cP \times^{Z^{der}} Z \ni (p,z) \mapsto p.z \in \cP'$$ which is obviously an isomorphism of $Z$-torsors and proves the claim.

Part (2) can be seen easily by appealing to the general Tannakian framework that $Z$-torsors over $\Spec \Q$ correspond to fibre functors $\omega: \Rep_\Q Z \rightarrow \vVec_\Q$, and that $\cP \times^{Z^{der}} Z = \cP'$ implies that we can factor $\omega_{\cP'}$ as
$$\omega_{\cP'}:\Rep_\Q Z \map{\Res} \Rep_\Q Z^{der} \map{\omega_\cP} \vVec_\Q.$$
Writing out what this statement actually means algebraically (and writing an arbitrary $\cO_Z$-comodule as a filtered colimit of finite dimensional ones), we recover (2).
\end{proof}

This has the technical advantage given by the following lemma. We thank Kestutis Cesnavicius for pointing out to us that the analogous result from Kisin's paper \cite[3.4.8]{kis2} where $G=Z$ is a general group of multiplicative type is false over the base $\Z[1/N]$ in general.
\begin{lem}
\label{finite torsor lemma}
Let $R$ be an integrally closed domain with fraction field $K$, and $G/R$ a finite group scheme. Then a torsor $T/R$ for $G$ is trivial iff $T_K$ is a trivial $G_K$-torsor.
\end{lem}
\begin{proof}
Since $G/R$ is finite, it is proper, which as a property stable under fpqc descent is inherited by $T$, and so $T(R) = \bigcap_{R_v \subset K}T(R_v) = T(K)$. In particular the latter is nonempty iff the former is.
\end{proof}

\subsubsection{}
\label{adjoint action}
With these remarks in mind, we can rewrite the action of $\cA_q(G)$ on the integral model $\cS_{K_q}/\cO_{E,v}$ explicitly following \cite[3.4.5]{kis2} as follows. A point $x \in \cS_{K_q}(T)$ gives rise to a triple $(A,\lambda,\epsilon^q)$ where $A/T$ is an abelian scheme up to prime to $q$-isogeny, $\lambda$ a weak polarisation of $A$ and $\epsilon^q$ a section of $\Gamma(T, \uIsom(V_{\A^{\infty,q}}, \hat{V}^q(A)))$.

By \cite[3.4.5]{kis2} and our previous remarks, if we take $(h,\gamma^{-1}) \in \cA_q(G)$ and $x$ associated to the triple $(A,\lambda,\epsilon^q)$ then the triple associated to $x.(h,\gamma^{-1})$ is isogenous to 
$$(A^\cP, \lambda^\cP, \epsilon^{q,\cP} \circ \tilde{\gamma} h \tilde{\gamma}^{-1})$$
where $\cP$ is the torsor for $Z^{der} \subset G^{der}$ given by the fibre of $\gamma$, $\tilde{\gamma}$ is an element of $G^{der}(F)$ for some finite Galois extension $F/\Q$ mapping to $\gamma$ under $G^{der}(F) \rightarrow G^{ad}(F)$, and the notations $A^\cP,\lambda^\cP, \epsilon^{q,\cP}$ are the ``twists by $\cP$'' as defined in \cite[3.1.3]{kis2}.

\subsubsection{}
We introduce the notation $\Z^N := \Z[1/N]$, to provide a slight simplification in situations where the notation quickly becomes messy.
Let (where plus notation denotes intersection with $G(\R)_+$,$G^{ad}(\R)^+$ as usually defined, and overline notation denotes closures in $\prod_{p|N}G(\Qp)$) 

$$\cA^N(G) := \frac{\prod_{p|N}G(\Qp)}{\overline{Z(\Z^N)}} *_{G(\Z^N)_+/Z(\Z^N)} G^{ad}(\Z^N)^+,$$
$$\cA^{N,\circ}(G) := \frac{\overline{G(\Z^N)_+}}{\overline{Z(\Z^N)}} *_{G(\Z^N)_+/Z(\Z^N)} G^{ad}(\Z^N)^+,$$

and (where here overline notation means closures in $G(\A^{\infty,q})$)

$$\tilde{\cA}^N_q(G) := \frac{\prod_{p|N}G(\Qp) \times \prod_{p\nmid qN} G(\Zp)}{\overline{Z(\Z^N)}} *_{G(\Z^N)_+/Z(\Z^N)} G^{ad}(\Z^N)^+ \subset \cA_q(G).$$

Let $K^N = \prod_{p \nmid N}G(\Zp)$, $K^{Nq} = \prod_{p \nmid Nq} G(\Zp)$ and
$$\Delta^N = \Ker(\cA^{N,\circ}(G) \rightarrow \cA^{N,\circ}(G^{ad})).$$

First, a group theoretic lemma following Deligne's Corvallis paper.
\begin{lem}
\label{ridiculous group lemma}
\begin{enumerate}
\item The group $\cA^{N,\circ}(G)$ is canonically the completion of $G^{ad}(\Z^N)^+$ with respect to the topology generated by the images of congruence subgroups of $G^{der}$ the form $K_N \times \prod_{p\nmid N} G^{der}(\Zp)$ as $K_N$ varies. 
In particular $\cA^{N,\circ}(G^{der}):=\cA^{N,\circ}(G)$ canonically depends only on $G^{der}$.
\item We can naturally identify
$$\cA^{N,\circ}(G^{der}) = \overline{G^{der}(\Z^N)_+}*_{G^{der}(\Z^N)_+} G^{ad}(\Z^N)^+$$ (where the closure is taken in $\prod_{p|N} G^{der}(\Qp)$).
\end{enumerate}
\end{lem}
\begin{proof}
For (1) we have a natural inclusion $G^{ad}(\Z^N)^+ \subset \cA^{N,\circ}(G),$ and it is easily checked that the image is dense. Moreover, a neighbourhood of the identity is $\frac{\overline{G(\Z^N)_+}}{\overline{Z(\Z^N)}}$ whose topology is generated by that of the congruence subgroups of $G$ with fixed hyperspecial level away from $N$. But by \cite[2.0.13]{del2} this topology is the same as that generated by the congruence subgroups of $G^{der}$ with fixed hyperspecial level away from $N$. Finally this neighbourhood of the identity is obviously complete, so we are done. From this description (2) follows immediately.

\end{proof}

Our key result is the following, whose proof closely follows that of \cite[3.4.6]{kis2}.

\begin{prop}
\label{delta free}
The group $\cA^N(G)$ acts naturally on the integral canonical model $\cS_{K^N}$, and $\Delta^N$ acts freely.
\end{prop}
\begin{proof}
 By the extension property, the first part can be checked on the generic fibre, where it follows from the self evident isomorphism
$$\tilde{\cA}^N_q(G) /K^{Nq} \rightiso \cA^N(G).$$
This (together with the compatibility with Kisin's construction) gives us the additional information that for $v|q, q\nmid N$, the action on $\cS_{K^N,v}$ can be described on the level of triples $(A,\lambda,\epsilon^q)$ where $\epsilon^q$ is now given modulo $K^{Nq}$.

Take $(h,\gamma^{-1}) \in \Delta^N \subset \overline{G^{der}(\Z^N)_+}*_{G^{der}(\Z^N)_+} G^{ad}(\Z^N)^+$, and $\tilde{\gamma} \in G^{der}(F)$ mapping to $\gamma \in G^{ad}(F)$ with $F/\Q$ finite Galois. We also let $\cP$ denote the $Z^{der}$-torsor of elements of $G^{der}$ mapping to $\gamma$. Since $(h,\gamma^{-1}) \in \Delta^N$ we see that $h\tilde{\gamma}^{-1} \in \prod_{p|N} Z^{der}(\Qp \otimes F)$.

Suppose $x \in \cS_{K^N}(\kappa)$ for some algebraically closed field $\kappa$ of characteristic $q$ or $0$, that $x.(h,\gamma^{-1}) = x$ and associated to $x$ is the triple $(A,\lambda,\epsilon^q/K^{qN})$. We need to show that $(h,\gamma^{-1})=1$.

As in the proof of \cite[3.4.6]{kis2} we can find a unique quasi-isogeny $\alpha: A \rightarrow A^{\cP}$ such that 
$$
\begin{CD}
    V \otimes \A^{\infty,q} \otimes F @>\tilde{\gamma} h \tilde{\gamma}^{-1}>>  V \otimes \A^{\infty,q} \otimes F@>\tilde{\gamma}^{-1}>>  V \otimes \A^{\infty,q} \otimes F \\
@V\epsilon^q VV        @V\epsilon^{q,\cP} VV @V \epsilon^q VV\\
    \hat{V}^q(A) \otimes F @>>\alpha_* \otimes 1>  \hat{V}^q(A^\cP) \otimes F @>>\iota_{\tilde{\gamma}}> \hat{V}^q(A) \otimes F.
\end{CD}
$$
commutes. This demonstrates that also $h\tilde{\gamma}^{-1} \in \Aut_\Q(A) \otimes F$. But the intersection of $\prod_{p|N} Z^{der}(\Qp \otimes F)$ and $\Aut_\Q(A) \otimes F$ inside $\prod_{p|N} \Aut_\Q(A)(\Qp \otimes F)$ is $Z^{der}(F)$, so we conclude that
$$h \tilde{\gamma}^{-1} \in Z^{der}(F).$$

As in \cite[3.4.6]{kis2} this demonstrates that $\cP$ is trivial as a $Z^{der}_\Q$-torsor over $\Q$. But $Z^{der}$ is a finite group, (\ref{finite torsor lemma}) implies that $\cP_{\Z[1/N]}$ (the torsor given by the inverse image of $\gamma$ in $G^{der}_{\Z[1/N]}$) is a trivial $\Z[1/N]$-torsor, so we can assume $\tilde{\gamma} \in G^{der}(\Z[1/N])_+$ and replacing $h$ by $h\tilde{\gamma}^{-1}$, assume $\gamma=1$. Moreover, we may take $F=\Q$, and so we deduce $h \in Z^{der}(\Q) = Z^{der}(\Z[1/N])$: that is to say, it is trivial as an element of $\Delta^N$, as required.
\end{proof}

We also need the following ingredient.\footnote{Note that we believe the corresponding result of \cite{kis2} is not true, so a result like this is perhaps also needed to deduce the extension property in that case.}

\begin{lem}
\label{delta finite}
The group $\Delta^N$ is finite.
\end{lem}
\begin{proof}
Let $\rho: G^{der} \rightarrow G^{ad}$ be the usual finite isogeny. By the discussion in \cite[2.0]{del2} there is a diagram with exact rows
$$\begin{CD}
   \overline{G^{der}(\Z^N)_+}  @>>> \cA^{N,\circ}(G^{der}) @>>> G^{ad}(\Z^N)^+ / \rho G^{der}(\Z^N)_+ \\
@VVV      @VVV @VVV\\
    \overline{G^{ad}(\Z^N)^+} @>>> \cA^{N,\circ}(G^{ad}) @>>> 1.
\end{CD}$$

Considering the kernels of the vertical maps, this puts $\Delta^N$ in an exact sequence between a subgroup of $Z^{der}(\A_N)$ and the group $G^{ad}(\Z^N)^+/\rho G^{der}(\Z^N)_+$. The former as a product of finite groups is visibly finite. The latter group is a subgroup of $G^{ad}(\Z^N)/\rho G^{der}(\Z^N)$, which is in turn a subgroup of $H^1_{fppf}(\Z^N,Z^{der})$, so it suffices to check that this is finite.

Since $Z^{der}$ is a finite group of multiplicative type, we may find a finite Zariski cover $U_i$ of $\Spec \Z[1/N]$ and finite \'etale covers $V_i \rightarrow U_i$ such that $Z^{der}|_{V_i}$ is isomorphic to a product of split finite multiplicative groups $\mu_k$. By the \v{C}ech to derived functor spectral sequence for this cover, we may reduce our claim to checking that for $L$ a number field and $k,M$ positive integers, the groups $H^1_{fppf}(\cO_L[1/M], \mu_k)$ are finite. But the Kummer exact sequence gives an exact sequence
$$1 \rightarrow \cO_L[1/M]^*/\cO_L[1/M]^{*k} \rightarrow H^1_{fppf}(\cO_L[1/M], \mu_k) \rightarrow \operatorname{Pic}(\cO_L[1/M]),$$
and both the outer terms are finite by classical algebraic number theory.
\end{proof}

\subsubsection{}
We now set out to prove the main theorem (\ref{S2 thm}) in the abelian type case. Let $(G_2,\fX_2)$ be a Shimura datum of abelian type, with $G_2/\Z[1/N]$ reductive. We first need to relate it to a datum of Hodge type, modifying \cite[3.4.13]{kis2} slightly.

\begin{lem}
\label{finding hodge}
Let $(H,\fY)$ be a Shimura datum of abelian type with $H$ adjoint. Then there exists a central isogeny $H' \rightarrow H$ such that whenever $(G,\fX)$ is of Hodge type with $(G^{ad},\fX^{ad}) \cong (H,\fY)$ then $G^{der}$ is a quotient of $H'$. 

Assume that $H$ is quasi-split and unramified at all $p\nmid N$. Then there exists a Shimura datum $(G,\fX)$ of Hodge type such that $(G^{ad},\fX^{ad}) \cong (H,\fY)$, $G^{der} = H'$ and $G$ is quasi-split and unramified at all $p \nmid N$.
\end{lem}
\begin{proof}
Most of this is proved in \cite[3.4.13]{kis2}: the only thing to check is that we can arrange for $G$ to be quasi-split and unramified at all $p \nmid N$. But since $H$ has this property, clearly $H'$ does, and so it suffices to control the centre, whose ramification is given by the totally imaginary quadratic extension $L/F$ of \cite[2.3.10]{del2} which can be chosen arbitrarily. In particular, we may take any prime $q|N$, construct $L$ by adjoining a $q$th (or $4th$ if $q=2$) root of unity and passing to a quadratic subfield. Then note that $L/F$ is unramified at all primes $v \nmid  q$, in particular at places over $p \nmid N$.
\end{proof}

\begin{lem}
\label{abelian isogeny integral}
Suppose we are given $G,G_2$ reductive over $\Q$ unramified away from $N$ and a central isogeny $f:G^{der} \rightarrow G_2^{der}$. Suppose we are also given a reductive model $\cG_2/\Z[1/N]$ of $G_2$. 

Then there exists a reductive model $\cG/\Z[1/N]$ of $G$ such that $f$ extends to
$$f:\cG^{der} \rightarrow \cG_2^{der}.$$
\end{lem}
\begin{proof}
We do the usual patching argument. Take any integral model $\cG/\Z[1/N]$ and note that there will be some $M$ such that $\cG[1/M]$ is reductive and $G^{der} \rightarrow G_2^{der}$ extends to $\cG^{der}[1/M] \rightarrow \cG^{der}_2[1/M]$.

By (\ref{glue groups}) we will be done if for every $p|M, p\nmid N$ we can find $G_p/\Zp$ a reductive model such that $f$ extends to $G_p^{der} \rightarrow \cG^{der}_{2,\Zp}$. But this is the case by the argument of \cite[3.4.14]{kis2}.
\end{proof}

\subsubsection{}
\label{abelian model construction}
Now given our $(G_2,\fX_2)$ of abelian type, it gives rise to an adjoint Shimura datum, which by (\ref{finding hodge}) is covered by $(G,\fX)$ of Hodge type and by (\ref{abelian isogeny integral}) we may take $G/\Z[1/N]$ reductive and $G^{der} \rightarrow G_2^{der}$ a central isogeny inducing a morphism $(G^{der},\fX^+) \rightarrow (G_2^{der},\fX_2^+)$ of connected Shimura data. Let $E = E(G,\fX) \subset \bar{\Q}$. Since for every $p \nmid  N$, $G$ and $G_2$ split over an unramified extension of $p$, we deduce that $E/\Q$ is unramified at all $p \not |N$, and by (\ref{component field}) we see their connected Shimura varieties at levels $K^N := G(\hat{\Z}^N)$ and $K_2^N := G_2(\hat{\Z}^N)$ are defined over $E_N/E$, the maximal abelian extension of $E$ unramified away from $N$. Let $\cO_N$ be its ring of integers, and note that $\cO_N[1/N] / \Z[1/N]$ is ind-\'etale.

Now, comparing our description (\ref{ridiculous group lemma} (1)) of $\cA^{N,\circ}(G^{der})$ with Deligne's description \cite[2.1.6]{del2} of the group acting on a connected Shimura variety, it is clear that $\cA^{N,\circ}(G)$ acts on $\Sh^+_{K^N}(G,\fX)_{E_N}$, and considering the subgroup
$$\Delta^N(G,G_2):=\Ker(\cA^{N,\circ}(G^{der}) \rightarrow \cA^{N,\circ}(G_2^{der})) \subset \Delta^N,$$
that the morphism $\Sh^+_{K^N}(G,\fX)_{E_N} \rightarrow \Sh^+_{K_2^N}(G_2,\fX_2)_{E_N}$ is given by taking the quotient by $\Delta^N(G,G_2)$. Moreover, by the previous section we have an integral model $\cS_{K^N}^+(G,\fX)_{\cO_N[1/N]}$ for $\Sh^+_{K^N}(G,\fX)_{E_N}$ satisfying the extension property, with a free (\ref{delta free}) action of the finite (\ref{delta finite}) group $\Delta^N$. We may therefore form the quotient by the finite subgroup $\Delta^N(G,G_2)$ and obtain a model $\cS^+_{K_2^N}(G_2,\fX_2)_{\cO_N[1/N]}$ for $\Sh^+_{K_2^N}(G_2,\fX_2)_{E_N}$ over $\cO_N[1/N]$. By (\ref{extension descent}) this model enjoys the extension property, and passing to finite levels we see it is smooth.

\subsubsection{}
Unlike at infinite or $K_p$-level, we do not know whether $\prod_{p|N} G_2(\Qp)$ acts transitively on $\pi_0(\Sh_{K^N}(G_2,\fX_2)_{\Qbar})$, so to conclude our proof we need an alternative to the usual `Deligne induction.' Noting that our argument up to this point holds for any model $\cG_2/\Z[1/N]$ for $G_2$, the following lemma is enough to conclude our argument. 

We will need the Shimura variety
$$\Sh(G,\fX) = \varprojlim_{K} \Sh_K(G,\fX)$$
at full infinite level, together with the usual fixed connected component
$\Sh^+_{\Qbar} \subset \Sh(G,\fX)_{\Qbar}$ containing the complex point $[x,1]$ for $x \in \fX^+$.

\begin{lem}
\label{assemble extension prop}
Let $(G,\fX)$ be any Shimura datum unramified away from $N$, $E \supset E(G,\fX)$, $K^N = \prod_{p\nmid N}G(\Zp)$ for some choice of integral model for $G$, and $E_N/E$ the maximal abelian extension unramified away from $N$.
\begin{enumerate}
\item Let $X^+_{K^N}$ be any component of $\Sh_{K^N}(G,\fX)_{\Qbar}$. Then $X^+_{K^N}$ is defined over $E_N$, and given for every choice of hyperspecial levels of the form $U^N = \prod_{p\nmid N} \cG(\Zp)$ for $\cG/\Z[1/N]$ a reductive model for $G$ a smooth integral model $\cS^+_{U^N}/\cO_N[1/N]$ for $\Sh^+_{U^N,E_N}$ with the extension property, we can construct a smooth integral model for $X^+_{K^N,E_N}$ with the extension property.
\item Given any smooth integral model for every $X^+_{K^N,E_N}$ with the extension property, their disjoint union gives a smooth integral canonical model for $\Sh_{K^N,E_N}$ which descends to $\cO_E[1/N]$ and to which the $\cA^N(G)$ action extends.
\end{enumerate}
\end{lem}
\begin{proof}
For (1), let $\pi: \Sh(G,\fX)_{\Qbar} \rightarrow \Sh_{K^N}(G,\fX)_{\Qbar}$ be the canonical projection, and take $a \in G(\Af)$ such that $\pi(\Sh^+.a) = X^+_{K^N}$. Note that this $a$ descends to an identification 
$$\Sh_{K^N}(G,\fX)_{\Qbar} \rightiso \Sh_{a^{-1}K^Na}(G,\fX)_{\Qbar}$$
defined over $E$ and under which $X^+_{K^N}$ is identified with $\Sh^+_{a^{-1}K^Na, \Qbar}$. 
Since the latter is defined over $E_N$ by (\ref{component field}), the former must be also. Moreover, since $a_p \in G(\Zp)$ for all but finitely many $p$, and in all other cases we are taking the conjugate of a hyperspecial subgroup, which are always hyperspecial and so have local reductive models, by (\ref{glue groups}) there exists a reductive model $\cG/\Z[1/N]$ for $G$ giving rise to the level $a^{-1}K^Na$. Hence by hypothesis we have a smooth integral model for $\Sh^+_{a^{-1}K^Na, E_N}$ with the extension property. Composing with the isomorphism induced by $a$ gives the required model for $X^+_{K^N,E_N}$.

For (2), we assume we are given for each $X^+_{K^N,E_N}$ some smooth integral model $\cX^+_{K^N}$ with the extension property. Letting $\cS_{K^N,\cO_N[1/N]}$ together with $$\iota: \cS_{K^N,\cO_N[1/N]} \otimes_{\cO_N[1/N]} E_N \rightiso \Sh_{K^N,E_N}$$ be their disjoint union, note that it still has the extension property. In particular the $\Gal(E_N/E)$-action on $\Sh_{K^N,E_N}$ extends to $\cS_{K^N,\cO_N[1/N]}$, and since $E_N/E$ is unramified away from $N$ this gives an \'etale descent datum from $\cO_N[1/N]$ to $\cO_E[1/N]$. Thus we may descend our model to $\cS_{K^N}/\cO_E[1/N]$ and by (\ref{extension localisation}) this model still has the extension property. The $\cA^N(G)$-action also extends to $\cS_{K^N,\cO_N[1/N]}$ by the extension property, and commutes with the $\Gal(E_N/E)$-action, so it descends, and for any $K_N \subset \prod_{p|N} G(\Qp)$ we have that $(\cS_{K^N}/K_N)\otimes \cO_N[1/N]$ is smooth, which is a property stable by fpqc descent and allows us to see that $\cS_{K^N}$ is a smooth canonical model.
\end{proof}

\section{Automorphic vector bundles and filtered $G$-bundles}
\subsection{Review of characteristic zero}
We sketch the main results of \cite[III]{milne3}, on which our results will build.

\subsubsection{}
\label{field grassmannian}
Let $(G,\fX)$ be a Shimura datum with reflex field $E$, and $\mu: \mathbb{G}_{m,\bar{E}} \rightarrow G_{\bar{E}}$ a Hodge cocharacter of $\fX$. Then we can form the \emph{compact dual} $Gr_{\mu}$ which represents the following functor. Fix a faithful representation $G \hookrightarrow GL(V)$ and tensors $s_\alpha \in V^\otimes$ such that $G$ is exactly the subgroup fixing these tensors, and note that $\mu$ induces a filtration $\text{Fil}^\bullet \subset V \otimes \bar{E}$. For $\pi:S \rightarrow \Spec \Q$ we let $\cV_S := \pi^*V$ be the constant vector bundle, which also carries tensors $s_{\alpha,S} = \pi^*s_\alpha$. Then as a functor on $E$-schemes
\begin{eqnarray*}
Gr_\mu(S) = & \{\text{filtrations } \cF^\bullet \text{ of } \cV_S \text{ such that for each geometric point }\\
& \bar{s} \in S, (V_{\bar{s}},\text{Fil}^\bullet_{\bar{s}}, s_\alpha) \cong (\cV_{S,\bar{s}}, \cF^\bullet_{\bar{s}},s_{\alpha,S})\}.
\end{eqnarray*}

This depends only on the conjugacy class of $\mu$ and is representable as an $E$-scheme. Moreover by construction it carries an algebraic action of $G_E$, which we write as a right action
$$Gr_\mu(S) \times G(S) \ni (\cF^\bullet, g) \mapsto g^{-1}(\cF^\bullet) \in Gr_\mu(S).$$

\subsubsection{}
\label{milne spb}
Let $Z_{nc} \subset Z(G)$ be the largest subtorus of $Z(G)$ split over $\R$ but with no subtorus split over $\Q$, and $G^c = G/Z_{nc}$. Then there is a $G^c$-torsor $P=P(G,\fX)$ over $\Sh(G,\fX)$ with an equivariant $G(\Af)$-action\footnote{It is natural to ask whether this extends to an equivariant action of Deligne's extension $G(\Af)/\overline{Z(\Q)} *_{G(\Q)_+/Z(\Q)} G^{ad}(\Q)^+$: in fact this group does act but not quite in a way that commutes with the algebraic $G^c$-action, as we shall later see.}, which can be easily defined analytically over $\C$ and by \cite[III,4.3]{milne3} admits a canonical model (in Milne's sense) over $E$, together with an integrable connection with regular singularities at infinity.

Moreover, via $G \rightarrow G^c$, $P(G,\fX)$ has a $G$-action, and there is a $G$-equivariant map \cite[III,4.6]{milne3} $\gamma: P(G,\fX) \rightarrow Gr_\mu$ which complex analytically is given by the Hodge filtration coming from $\fX$, but is algebraic and descends to $E$.

To begin discussing any functoriality of this construction we need a lemma.
\begin{lem}
\label{c functorial}
Let $f:(G_1,\fX_1) \rightarrow (G_2,\fX_2)$ be a morphism of Shimura data. Then there is an induced map
$$G_1^c \rightarrow G_2^c.$$
\end{lem}
\begin{proof}
This comes down to showing that $f(Z(G_1)_{nc}) \subset Z(G_2)$. Suppose for contradiction we have $\R^* \cong \Psi \subset Z(G_1)(\R)$ with $f(\Psi)$ intersecting trivially with $Z(G_2)(\R)$. Let $h \in \fX_1$, and recall that $\ad f(h(i))$ is a Cartan involution acting on $G_2^{ad}$, so the real group
$$H = \{g \in G_2^{ad}(\C): f(h(i)) g f(h(i))^{-1} = \bar{g}\}$$
is compact. On the other hand for any $\psi \in \Psi$ we have
$$h(i)\psi h(i)^{-1} = \psi = \bar{\psi}$$
since $\psi$ is central in $G_1$, $h(i) \in G_1(\R)$, and $\psi$ is real.
Hence, we have an embedding of $\Psi \cong \R^*$ into the compact group $H(\R)$, which is absurd.
\end{proof}
Note that there is no such functoriality for general group morphisms. For example letting $F$ be a totally real field of degree $d$ acting on itself by multiplication we get a morphism
$$F^\times \hookrightarrow GL_d$$
failing to have the required property for all $d>1$.

\subsubsection{}
\label{avbs}
Automorphic vector bundles are typically parameterised by complex representations of the parabolic subgroup $P_\mu$ associated with the Hodge cocharacter $\mu$ in the usual fashion (for example one may define $P_\mu \subset G_\C$ as the subgroup preserving the filtrations $\mu$ induces on $\Rep G_\C$). By a complex analytic interpretation of the Grassmannian it is easy to show these are in correspondence with $G_\C$-equivariant vector bundles on $Gr_{\mu,\C}$. We therefore take as our input data a $G$-equivariant vector bundle $\cJ$ on $Gr_{\mu}$ defined over $L/E$ some number field, and we assume the $G$-action factors through $G^c$.\footnote{This is a reasonable condition to impose since $Z \subset P_\mu$ acts trivially on $Gr_{\mu}$.}

Given this data, we can pull it back along $\gamma$ to get a $G(\Af) \times G^{c}$-equivariant vector bundle on $P(G,\fX)_L$ and therefore a $G(\Af)$-equivariant vector bundle $\cV(\cJ)$ on $\Sh(G,\fX)_L$. This is the construction of canonical models for automorphic vector bundles we seek to perform integrally.

\subsubsection{}
We need a basic functoriality property for which we could not find a direct reference but which follows easily from Milne's definition together with the map on complex points induced by $\fX \times G(\Af) \times G^c(\C) \rightarrow \fX' \times G'(\Af) \times G'^{c}(\C)$ which we note relies on (\ref{c functorial}).

\begin{lem}
\label{spb functorial}
Let $f:(G,\fX) \rightarrow (G',\fX')$ be a morphism of Shimura data, $\mu,\mu'$ Hodge cocharacters of $\fX,\fX'$ respectively. Then there is a diagram
$$
\begin{CD}
    \Sh(G,\fX) @<<< P(G,\fX) @>>> Gr_{\mu}\\
@VVV        @VVV @VVV\\
     \Sh(G',\fX') @<<<  P(G',\fX') @>>> Gr_{\mu'}
\end{CD}
$$
 defined over 
$E(G,\fX)$ which is $G$-equivariant and $G(\Af)$-equivariant in the obvious senses.
\end{lem}

\subsubsection{}
In the case where $G=T$ a torus we also need the following, the main content of which is due to Blasius. Let $p$ be a prime containing a place $v$ of $E \supset E(T,h)$, $\omega_{et}$ the usual fibre functor giving \'etale local systems on the Shimura variety, and fix $\overline{E_v}$ an algebraic closure of $E_v$ letting $\Gamma_{E_v}:=\Gal(\overline{E_v}/E_v)$. Consider the fibre functor coming from $p$-adic Hodge theory $$\omega_{v,dR}: \Rep_{\Qp}(T^c) \ni V \mapsto (\omega_{et}(V) \otimes_{\Qp} B_{dR})^{\Gamma_{E_v}} \in \vVec_{E_v}.$$

\begin{prop}
\label{dr identification}
Suppose $T$ is split by a CM field, and let $X = \Spec E \subset \Sh_U(T,h)$ be a component of the Shimura variety for $U \subset T(\Af)$ open compact. There is a natural isomorphism between $\omega_{P,X,v}: V \mapsto (V \times^{T^c} P_U(T,h)|_{X}) \otimes_E E_v$ and $\omega_{v,dR}$.
\end{prop}
\begin{proof}
By (\ref{spb functorial}) it suffices to do the case $T=T^c$. We observe that any cocharacter $\mu: \mathbb{G}_{m,E} \rightarrow T^c_E$ has weight defined over $\Q$. Indeed, by definition $X_*(T^c)_{\Q}$ is the summand of $X_*(T)_{\Q}$ on which either Galois acts trivially or complex conjugation acts via $-1$, so for any $\mu \in X_*(T^c)$, $\mu + \mu^c$ lands in the summand on which Galois acts trivially.

Combining this with fact that $T$ hence $T^c$ is split by a CM field, the induced Shimura datum $(T^c,h^c)$ is of CM type, and thus admits a characterisation as a moduli space of CM motives $M$, which we take as those for which the Betti fibre functor $\omega_B = H_B(M(\rho)):\Rep_\Q T \rightarrow \vVec_\Q$ is trivial.

Therefore a point $x \in X(\Qbar)$ has attached to it a $T$-valued CM motive $M: \Rep_{\Q}(T) \rightarrow (CM/\Qbar)$, and the fibre functor attached to $x^*P_U(T,h)$ is given by $\omega_P(\rho) := H_{dR}(M(\rho)/\Qbar).$ Noting that for $\sigma \in \Gal(\Qbar/E)$ we have a canonical isomorphism $$H_{dR}(M^\sigma(\rho)/\Qbar) \cong H_{dR}(M(\rho)/\Qbar) \otimes_{\sigma} \Qbar,$$ we see that $\omega_P$ carries a canonical descent datum to $X=\Spec E$ which defines $P_U(T,h)|_X$. We may also use the reciprocity law to canonically identify the $p$-adic etale cohomology fibre functors 
$$H_{et} \circ M = \omega_{et}: \Rep_{\Qp}(T_{\Qp}) \rightarrow \Bun_{\Qp},$$ which in particular have an equivariant $\Gamma_E$-action coming from $\omega_{et}$.

Now fix an embedding $\Qbar \hookrightarrow \overline{E_v}$, and a faithful representation $T \hookrightarrow GL(V)$, and $s_\alpha$ cycles on $V^\otimes$ fixed precisely by $T$. For any $M \in X(\Qbar)$, these give rise to absolute Hodge cycles on $H_B(M(V))^\otimes$, which in turn give rise to de Rham cycles $s_{\alpha,dR} \in H_{dR}(M(V))^\otimes$ via the Betti-de Rham comparison, and \'etale cycles $s_{\alpha,et} \in \omega_{et}(V)^\otimes$. These $s_{\alpha,et}$ are $\Gamma_E$-invariant because the action factors
$$\Gamma_E \rightarrow U_p \subset T(\Qp) \rightarrow GL(\omega_{et}(V))$$
by construction, and as in \cite[2.2.1]{kis2} this implies the $s_{\alpha,dR} \in \omega_{P,X,v}(V)^\otimes$ because an absolute Hodge cycle is determined by either component.

Let us fix $Y/L$ with $\Qbar \supset L$ a finite Galois extension of $E$ such that $M(V)$ is realised in the cohomology of $Y$, and a place $w|v$ of $L$ determining $L_w \subset \overline{E_v}$. By Blasius' theorem on de Rham cycles \cite{blasius}, the $p$-adic Hodge theoretic comparison map
$$H_{et}(M(V),\Qp)^\otimes \otimes_{\Qp} B_{dR} \rightiso H_{dR}(M(V)_{L_w}/L_w)^\otimes \otimes_{L_w} B_{dR}$$
identifies $s_{\alpha,et}$ with $s_{\alpha,dR}$. Unravelling the definitions, we see that $$P_U(T,h)_{L} = \uIsom_{s_\alpha}(V_L,H_{dR}(M(V))_L)$$ and 
$$P_{\omega_{v,dR}} \otimes L_w = \uIsom_{s_\alpha}(V_{L_w},(\omega_{et}(V) \otimes_{L_w} B_{dR})^{\Gamma_{L_w}}).$$ 
Thus putting it all together we get a canonical\footnote{The choice of $s_\alpha$ does not affect the identification. To see this note that it does not change under adding more tensors so given two sets of choices one can just compare both with the union.} identification $\omega_{P,X,v} \otimes L_w \cong \omega_{v,dR} \otimes L_w.$

We conclude by checking that this descends to $E_v$. Indeed, $\Gamma_{E_v}$ acts canonically on both sides of the $p$-adic comparison map compatibly, so we have an isomorphism of $\Gal(L_w/E_v)$-modules
$$(\omega_{et}(V) \otimes_{\Qp} B_{dR})^{\Gamma_{L_w}} \rightiso H_{dR}(M(V)_{L_w}/L_w) = \omega_{P,X,v}(V) \otimes_{E_v} L_w$$
and taking $\Gal(L_w/E_v)$-invariants therefore $\omega_{v,dR}(V) = \omega_{P,X,v}(V)$. Finally, we have already observed that the $s_{\alpha,et}$ and $s_{\alpha,dR}$ are Galois invariant, so these too descend to $E_v$.
\end{proof}

\subsection{Moduli of $\mu$-filtrations of $G$}
In \S3.2 and \S3.3 we make a brief digression from the theory of Shimura varieties to discuss Grassmannians and filtrations more generally. Let $R$ be a domain with fraction field $K$ of characteristic zero, $R'/R$ an \'etale cover, and $G/R$ a connected reductive group. 
\subsubsection{}
Suppose we are given a cocharacter $\mu: \mathbb{G}_{m,R'} \rightarrow G_{R'}$. This cocharacter gives us a parabolic subgroup $P_\mu \subset G_{R'}$, which we can view as a point
$x_\mu \in \Par_{G/R}(R')$, where we recall \cite[5.2.9]{con} that $\Par_{G/R}$, the functor assigning to each $R'/R$ the set of parabolic subgroups of $G_{R'}$ defined over $R'$, is a proper smooth scheme over $R$.

After making a base change $R \subset K \rightarrow \bar{K}$ we get a well-known finite decomposition 
$$\Par_{G/R} \otimes_R \bar{K} = \Par_{G_{\bar{K}}/\bar{K}} = \coprod_i G_{\bar{K}}/P_i$$
where the $P_i$ are representatives of the finitely many $\bar{K}$ conjugacy classes of parabolic subgroups of $G_{\bar{K}}$.

For $P_\mu \subset G_{R'}$ a parabolic subgroup, we denote its conjugacy class by $[P_\mu]$, in a precise sense we will soon make clear. Let us say that \emph{$[P_\mu]$ is defined over $R$} if there is a component\footnote{Here and in much of what follows we use ``component'' in the relative sense of a `component' of a morphism $X \rightarrow Y$ being a $Y$-subscheme $X'\subset X$ such that the preimage of any connected open subscheme $U\subset Y$ is connected.} $Z_\mu \subset \Par_{G/R}$ defined over $R$ such that $x_{\mu} \in Z_\mu(R')$ and $Z_{\mu,\bar{K}}$ is connected.

This definition in a sense is saying that ``being \'etale locally conjugate to $P_\mu$'' is a notion that is defined over $R$, even if $P_\mu$ itself is not. More precisely, we have the following.

\begin{lem}
Let $\mu$ be a cocharacter as above defined over $R'$ such that $[P_\mu]$ is defined over $R$. For $S$ an $R$-algebra, and $P \subset G_S$ a parabolic subgroup, $P$ is \'etale-locally conjugate to $P_\mu$ if and only if $x_P \in \Par_{G/R}(S)$ factors through $Z_{\mu}$.
\end{lem}
\begin{proof}
The statement may be checked \'etale locally, so we may work over $R'$, at which point by the construction of $\Par_{G/R}$, $G_{R'}/P_\mu \subset Z_{\mu,R'}$ is a union of components. Working over $R'\otimes_R \bar{K}$ using the fact that the fibres of $Z_{\mu}$ over each generic point are connected, we see that in fact $G_{R'}/P_\mu = Z_{\mu,R'}$, so the desired statement follows immediately from \cite[5.2.8]{con}.
\end{proof}

For us an important context where the above holds will be the following.

\begin{prop}
\label{grass integrality}
Assume $R$ is a Dedekind domain with $K=\Frac R$ of characteristic zero, $G/R$ connected reductive and $L/K$ a finite extension. Suppose $\mu: \mathbb{G}_{m,L} \rightarrow G_L$ is a cocharacter whose conjugacy class is defined over $K$. 

Then there exists an \'etale cover $R'/R$ and a cocharacter $\mu': \mathbb{G}_{m,R'} \rightarrow G_{R'}$ such that if we write $R'\otimes_R K \cong \prod_i L_i$, there are embeddings $L \hookrightarrow L_i$ such that $\mu' \otimes_R K$ is $G(R' \otimes_R K)$-conjugate to $\mu \otimes_L \prod_i L_i$.

Furthermore, this $\mu'$ has the property that $[P_{\mu'}]$ is defined over $R$ and independent of the choice of $\mu'$.
\end{prop}
\begin{proof}
For existence, take $R'/R$ an \'etale cover which splits $G$ and whose generic points contain $L$, and fix embeddings $L \hookrightarrow L_i$. Let $T \subset G_{R'}$ be a split maximal torus and let $T_\mu \subset G_{R'\otimes_R K}$ be a maximal torus containing the image of $\mu$ which perhaps after enlarging $R'$ we may assume is also split. Then there exists $g \in G(\prod_{i}L_i) = \prod_i G(L_i)$ such that $gT_Kg^{-1} = T_\mu$ and so 
$$\mu' = g^{-1}\mu g \in \Hom_{R'\otimes_R K}(\Gm,T_K) = \Hom_{R'}(\Gm,T)$$
is a cocharacter with the desired property.

Let us next show that $[P_{\mu'}]$ is defined over $R$. By definition we see that $x_\mu$ and $x_{\mu'}$ lie on the same component of $\Par_{G/R} \otimes (R' \otimes_R K)$, whence certainly on that of $\Par_{G/R} \otimes R'.$
But since the conjugacy class of $\mu$ is defined over $K$, $x_{\mu}$ lies on a component $Z_{\mu,K} \subset \Par_{G_K/K}$ which is geometrically connected. Letting $Z_{\mu}$ be its closure in (i.e. the corresponding component of) $\Par_{G/R}$, we obtain our witness to the fact that $x_{\mu'}$ is a defined over $R$. Also since $Z_{\mu}$ is determined by $\mu$, in particular it does not depend on $\mu'$.
\end{proof}

\subsubsection{}
Let $S/R$ be a scheme. We let $\Bun_S$ denote the category of vector bundles (projective finitely generated modules) on $S$, and $\Rep_R(G)$ the category of algebraic representations $G \rightarrow \Aut(V)$ for $V \in \Bun_R = \Bun_{\Spec R}$.

A \emph{filtered bundle} over $S$ is a vector bundle $M/S$ together with a decreasing complete exhaustive filtration $F^\bullet \subset M$ by flat submodules such that $\gr_F^\bullet M := \bigoplus_p F^p/F^{p+1}$ is flat over $S$. These form an exact category $\Fil_S$, allowing us to define a \emph{filtered $G$-bundle} over $S$ to be a faithful exact tensor functor 
$$\cF: \Rep_{R}(G) \rightarrow \Fil_S.$$

We say that $\cF$ is a \emph{filtration of $G$ over $S$} (or just ``filtration of $G$'' if no confusion will arise) if its composite with the forgetful functor $\Fil_S \rightarrow \Bun_S$ is naturally isomorphic to the usual forgetful functor $\Rep_R(G) \rightarrow \Bun_R \map{\otimes \cO_S} \Bun_S$.

Let $\mu: \mathbb{G}_{m,R'} \rightarrow G_{R'}$ be a cocharacter defined over some \'etale cover $R'/R$ such that $[P_{\mu}]$ is defined over $R$. It induces a grading on each $V \in \Rep_R(G) \otimes R'$, each of which in turn gives such $V$ the structure of a filtered bundle, so we can define a canonical filtration of $G$ associated with $\mu$ 
$$\cF_{\mu}: \Rep_R(G) \map{\Fil \circ \mu_*} \Fil_R(G) \otimes R'.$$

\subsubsection{}
Let us say that a filtration $\cF$ of $G$ over $S$ is a \emph{$\mu$-filtration} if \'etale locally on $S$ we have $\cF \cong \cF_\mu$. We make the necessary remark that of course given $f:S' \rightarrow S$, whenever $\cF$ is a $\mu$-filtration of $G$ over $S$, we may take an \'etale cover $\{U_\alpha \rightarrow S\}$ witnessing that $\cF$ is a $\mu$ filtration and pulling it back along $f$ it will give an \'etale cover of $S'$ witnessing that $f^* \circ \cF$ is a $\mu$-filtration of $G$ over $S'$.

Thus we have a natural functor on $R'$-schemes
$$Gr_{\mu,R'}(S) = \{\mu\text{-filtrations of }G\text{ over }S\}.$$

\begin{prop}
This has the following properties.
\begin{enumerate}
\item The functor $Gr_{\mu,R'}$ is representable by a smooth proper $R'$-scheme which in fact canonically descends to $Gr_{\mu}/R$.
\item The scheme $Gr_{\mu}/R$ comes equipped with a natural $G$-action.
\item If $R$ is a field, this agrees with the construction of (\ref{field grassmannian}).
\end{enumerate}
\end{prop}
\begin{proof}
In the usual fashion $\mu$ determines a parabolic $P_\mu \subset G_{R'}$, and we claim that $Gr_{\mu,R'} \cong G/P_\mu$ which in turn is smooth and projective by \cite[5.2.8]{con}, and canonically descends by the definition of $[P_\mu]$ being defined over $R$.

Loc. cit. it is also shown that $G/P_\mu$ represents the functor of subgroups of $G$ which are \'etale-locally conjugate to $P_\mu$, so it suffices to check this coincides with our functor. Given a $\mu$-filtration $\cF$ over $S$ we can define the subgroup $P_\cF \subset G$ of elements which preserve $\cF$ (if you like, acting on all representations). Since $\cF$ is a $\mu$-filtration, \'etale locally there is an identification of $\cF$ with $\cF_\mu$, which conjugates $P_\cF$ onto $P_\mu$. Thus $\cF \mapsto P_\cF$ gives a map $Gr_\mu(S) \rightarrow G/P_\mu(S).$

Let us construct an inverse. Given $P/S$ \'etale locally conjugate to $P_{\mu,S}$, after passing to an \'etale cover $S' \rightarrow S$ there exists $g \in G(S')$ such that $P_{S'} = g P_{\mu,S'} g^{-1}$. In particular $P_{S'}$ is a parabolic subgroup of $G_{S'}$, and letting $\mu_P:= g \circ \mu$ we see that $P_{S'}$ preserves the filtration $\cF_P$ defined by $\mu_P$ of $G$ over $S'$. We must now check this filtration is independent of the choice of $g$ and so in particular is canonical and descends to $S$. 

Suppose we take $h \in G(S')$ such that $hP_{\mu,S'}h^{-1} = gP_{\mu,S'}g^{-1} = P_{S'}$. Then
$$hg^{-1}P_{S'}gh^{-1} = hP_{\mu,S'} h^{-1} = P_{S'}$$
so $hg^{-1} \in N_G(P)(S') = P(S'),$
where the final equality is again by \cite[5.2.8]{con}. It follows that $g\mu$ and $h\mu$ induce the same filtration.

Part (2) is now obvious, since $G$ acts by conjugation on $G/P_\mu$, and part (3) is immediate from \cite[5.2.7 (1)]{con} and an easy verification shows that the $G$-actions agree.
\end{proof}

\subsubsection{}
In the context of (\ref{grass integrality}) where we are given a conjugacy class of cocharacters $\mu:\mathbb{G}_{m,L} \rightarrow G_L$ defined over $K=\Frac(R)$ and deduce the existence of a $\mu'$ defined over an \'etale cover of $R$ inducing a conjugacy class of parabolics defined over $R$ and hence a $Gr_{\mu'}/R$ we use the notation $\mathcal{GR}_{\mu} := Gr_{\mu'}$ noting the canonical identification
$$\mathcal{GR}_{\mu} \otimes_R K \cong Gr_{\mu}.$$

\begin{lem}
\label{gr functoriality}
Suppose $G_1 \rightarrow G_2$ is a map of connected reductive groups over $R$ and $\mu_1: \Gm \rightarrow G_{1,R'}$ giving a conjugacy class defined over $R$ and inducing $\mu_2: \Gm \rightarrow G_{1,R'} \rightarrow G_{2,R'}$. There is a natural map
$$Gr_{\mu_1} \rightarrow Gr_{\mu_2}.$$
\end{lem}
\begin{proof}
Given $S/R$ and $\cF \in Gr_{\mu_1}(S)$, recall that $\cF$ is specified by a fibre functor
$$\Rep_R(G_1) \rightarrow \Fil_S.$$
Composing with the restriction map $\Rep_R(G_2) \rightarrow \Rep_R(G_1)$, we get a new fibre functor from $\Rep_R(G_2)$ which it is easy to check gives an element of $Gr_{\mu_2}(S)$, defining the map required.
\end{proof}

\subsection{Filtered $G$-bundles}
Fix $G,\mu$ as above and suppose we have $X/R$ a scheme and $P \rightarrow X$ a $G$-bundle on $X$. A $\mu$-\emph{filtration} of $P$ is a $G$-equivariant map of $R$-schemes $$\gamma:P \rightarrow Gr_\mu.$$

\begin{lem}
To give a $\mu$-filtration $\gamma$ on $P$ is to give a fibre functor
$$\omega^{\gamma}_P: \Rep_{R}(G) \rightarrow \Fil_X$$
which \'etale locally is isomorphic to $\cF_\mu$ and such that the composite with the forgetful functor
$$\Rep_{R}(G) \map{\omega^{\gamma}_P} \Fil_X \rightarrow \Bun_X$$
is equal to the fibre functor $\omega_P$ defined by $P$.
\end{lem}
\begin{proof}
Suppose we are given a $\mu$-filtration $\gamma:P \rightarrow Gr_{\mu}$ of $P$. Pulling back the universal $\mu$-filtration of $G$, we obtain a $G$-equivariant $\mu$-filtration of $G$ over $P$, which descends to a $\mu$-filtration of $G$ over $X$. I.e. we obtain a fibre functor $$\omega_P^\gamma:\Rep_R(G) \rightarrow \Fil_X.$$ Since it comes from a $G$-equivariant $\mu$-filtration of $G$ over $P$, whose forgetful functor to $\Bun_P$ by definition is the canonical one from $\Rep_R(G)$ which descends to $\omega_P$, we see that it satisfies the condition in the lemma.

Conversely, if we are given $\omega^\gamma_P$ satisfying the condition, we can pull it back along $P \rightarrow X$ to obtain a $G$-equivariant $\mu$-filtration on $P$, which is the same as a $G$-equivariant map $\gamma: P \rightarrow Gr_\mu$.
\end{proof}

We also need the following criterion for extending $\mu$-filtrations over $K$ to $\mu$-filtrations over $R$.

\begin{lem}
\label{filtration extension}
Suppose we are given $\mu:\mathbb{G}_{m,L} \rightarrow G_L$ for $L/K=\Frac(R)$ a finite extension whose conjugacy class is defined over $K$, $X/R$ a scheme and $P \rightarrow X$ a $G$-bundle. A $\mu$-filtration
$$\gamma: P_K \rightarrow Gr_{\mu}$$
extends to a $G$-equivariant
$$P \rightarrow \mathcal{GR}_{\mu}$$
if for $\rho: G \hookrightarrow GL(V)$ a faithful representation on $V/R$ finite projective,
the bundle $\cV := V \times^G P/X$ carries a filtration $\cV^\bullet$ (making $\cV$ into a filtered bundle in the above sense) extending that of $\cV_K = \omega^\gamma_{P_K}(V)$.
\end{lem}
\begin{proof}
Let $R'/R$ be an \'etale cover and $\mu: \mathbb{G}_{m,R'} \rightarrow G_{R'}$ a cocharacter conjugate to $\mu$ as in (\ref{grass integrality}). We let $Gr_{\mu'}^G$ and $Gr_{\mu'}^{GL(V)}$ be the two Grassmannians formed from considering $\mu':\Gm \rightarrow G$ and $\mu':\Gm \rightarrow G \map{\rho} GL(V)$ respectively, both defined over $R$.

By (\ref{gr functoriality}) we obtain a map $$\rho_*: Gr_{\mu'}^G \rightarrow Gr_{\mu'}^{GL(V)}.$$
We claim this map is a closed immersion, and this suffices to prove the lemma because we are assuming that $\rho_*(\gamma) \in Gr_{\mu'}^{GL(V)}\otimes_R K$ extends to a map $\tilde{\gamma}: P \rightarrow Gr_{\mu'}^{GL(V)}$, and since $Gr_{\mu'}^{GL(V)}$ is flat that the ideal sheaf of $Gr_{\mu'}^G$ is killed by $\tilde{\gamma}^*$ may be checked on the generic fibre.

Since it is clearly proper (as both the source and target are proper), it suffices to show it is a monomorphism, for which it suffices to check the functor of points is injective. But this is obvious: given two $\mu'$-filtrations of $G_T$ for a test scheme $T$, if they induce the same filtration on $V_T$ then they are equal, since any other representation of $G_T$ can be embedded in a tensor construction on $V_T$ and will have to receive the induced filtration.

\end{proof}
\section{Integral Models for the standard principal bundle}
\subsection{Breuil-Kisin modules and lattices in de Rham cohomology}

We consider the following adaptation of the results of \cite{kis1}, as packaged in \cite[1.2]{kis2}. Fix $F = \prod_{i=1}^s F_i$ some finite \'etale algebra over $\Qp$, and let $\kappa_i$ be the residue field of $F_i$, and $E_i(u)$ the monic minimal polynomial of a uniformiser $\varpi_i$ for $F_i$ over $W(\kappa_i)$. Consider the ring $\fS = \fS_F := \prod_i W(\kappa_i)[[u]]$ and $E(u) = (E_1(u),\dots,E_s(u))$. We equip $\fS$ with the Frobenius $\varphi$ raising $u \mapsto u^p$ and the canonical Frobenius on each $W(\kappa_i)$.

Define the category $\text{Lisse}^{crys}_{\Zp}(F)$ to be that of crystalline constant rank lisse $\Zp$-sheaves on $\Spec F$, i.e. Galois-stable $\Zp$-lattices in tuples of crystalline $p$-adic representations $(\sigma_i: \Gal(\bar{F_i}/F_i) \rightarrow GL_{n}(\Qp)).$ Define the category $\text{Mod}_\fS^\varphi$ of $\fS$\emph{-modules} to consist of finite free $\fS$-modules $\fM$ together with a $\varphi$-semilinear isomorphism
$$1 \otimes \varphi: \varphi^*(\fM)[1/E(u)] \rightiso \fM[1/E(u)].$$

\begin{thm}
\label{kismod}
There is a fully faithful tensor functor
$$\fM: \Lisse^{crys}_{\Zp}(F) \rightarrow \Mod^\varphi_{\fS_F}$$
compatible with the formation of symmetric and exterior powers, unramified base change $F \rightarrow F'$ of finite \'etale algebras over $\Qp$, and with the property that
$$\D(L) := \varphi^*\fM(L) \otimes_{\fS_F} \cO_F \subset D_{dR}(L \otimes \Qp)$$
obtained by tensoring along the map $u \mapsto (\varpi_i)$ is a natural $\cO_F$-lattice in $D_{dR}(L \otimes \Qp)$.
Moreover, if $\cA$ is an abelian variety over $\cO_F$, and $L = T_p(\cA_F)^*$, then this lattice is identified with integral de Rham cohomology
$$
\begin{CD}
     H^1_{dR}(\cA/\cO_F) @>\subset >> H^1_{dR}(\cA_F/F) \\
@|        @VV\cong V\\
     \D(L) @>\subset >>  D_{dR}(L\otimes \Qp).
\end{CD}
$$
\end{thm}
\begin{proof}
The first part is just \cite[1.2.1]{kis2} and the second follows perhaps most quickly from \cite[1.8 (ii)]{bms} (although since we are in the case of abelian varieties the theorem probably also can be deduced directly from the theory of Breuil and Kisin).
\end{proof}

\subsubsection{}
For our application to integral models of Shimura varieties over $\cO_E[1/N]$, we also need the following abstract lemmas. We thank one of the anonymous referees for pointing out the ideas for the argument for (\ref{lattice uniqueness}) at \cite[6.15]{maul}.
\begin{lem}
\label{lattice uniqueness}
\begin{enumerate}
\item Let $A$ be a Dedekind domain with fraction field $K$, $X/A$ a smooth scheme and suppose we are given two vector bundles $\cL_1,\cL_2$ over $X$ together with an identification
$$\theta: \cL_1 \otimes_A K \rightiso \cL_2 \otimes_A K.$$
Suppose further that $\theta$ extends as an isomorphism to the formal completion of $X$ at all maximal ideals of $A$. Then $\theta$ extends over $X$.
\item Let $W = W(k)$ for $k$ a perfect field with fraction field $K$. Suppose $X/W$ is a smooth scheme, with $p$-adic completion $\hat{X}$, and special fibre $X_0$, and suppose we are given vector bundles $\cL_1,\cL_2$ over $\hat{X}$ together with an identification
$$\theta: \cL_1[1/p] \rightiso \cL_2[1/p].$$
Suppose further that there is a Zariski dense subset $U_0 \subset X_0$ of the special fibre such that each $x_0 \in U_0$ admits a lift $\tilde{x}_0 \in \hat{X}(W(\kappa(x_0)))$ with the property that $\tilde{x}_0^*(\theta)$ extends to an isomorphism  $\tilde{x}_0^*(\cL_1) \rightiso \tilde{x}_0^*(\cL_2)$.
Then $\theta$ extends over $\hat{X}$.
\end{enumerate}
\end{lem}

\begin{proof}
Note that in both cases the claim may be checked locally on $X$, so we may assume $\cL_1,\cL_2$ are both free and that $X = \Spec R$ is affine, and by picking bases and considering the matrices of both $\theta$ and $\theta^{-1}$ the question can be reduced to a question about the matrix coefficients.

For (1), we wish to show matrix values lying in $R \otimes_A K$ in fact lie in $R$. The matrix values all lie in $R[1/D]$ for some $D \in A$, since $R$ is of finite type over $A$, and the $\cL_i$ of finite rank. But by the hypothesis for each maximal ideal $\fp$ containing $D$, we know $\theta$ and $\theta^{-1}$ extend over $\hat{X}_{/\fp}$, implying that the matrix values also lie in $R\otimes_A A_{\fp}$, hence they lie in $R$ as required.

For (2), letting $\hat{R}$ be the $p$-adic completion of $R$, we are required to show that a matrix coefficient $f \in \hat{R}[1/p]$ in fact lies in $\hat{R}$. Suppose for contradiction it does not, and let $r>0$ be minimal such that $F=p^rf \in \hat{R}$. The hypothesis tells us that for any $x_0 \in U_0$ and some lift $\tilde{x}_0$ we have that $\tilde{x}_0^*(f) \in W(\kappa(x_0))$. Therefore $\tilde{x}_0^*(F) \in p^r W(\kappa(x_0))$ and in particular $x_0^*(F)=0$, or $F \in \fm_{x_0} \subset R \otimes_W k$. Since the set of such $x_0$ is dense, and $R \otimes_W k$ a reduced algebra of finite type over a field, we deduce that the $F=0$ restricted to the special fibre. I.e. $F=pF'$ for some $F' \in \hat{R}$. But then $p^{r-1} f = F' \in R$ giving the required contradiction.
\end{proof}

\begin{lem}
\label{lattice torsors}
Let $R$ be a PID with fraction field $K$, $X/R$ a flat scheme, $G/R$ a flat affine group scheme of finite type, and $P/X_K$ a $G_K$-torsor. Suppose we have two pairs $(\cP_i,\iota_i)$ $i=1,2$ where $\cP_i$ is a $G$-torsor over $X$ and $\iota_i:\cP_i \otimes_R K \rightiso P$ an isomorphism of $G_K$-torsors over $X_K$.

Then the map $\iota_2^{-1} \iota_1: \cP_1 \otimes K \rightiso \cP_2 \otimes K$ extends over $R$ if and only if for every representation $G \rightarrow GL(V)$ with $V/R$ finite free the composite isomorphism
$$\omega_{\cP_1}(V) \otimes K \map{\iota_{1*}} \omega_{P}(V_K) \map{\iota_{2*}^{-1}} \omega_{\cP_2}(V) \otimes K$$
identifies the lattices $\omega_{\cP_1}(V)$ and $\omega_{\cP_2}(V)$. 
\end{lem}
\begin{proof}
Recall the natural equivalence \cite[1.2]{brosh} under our hypotheses between the groupoid of $G_X$-torsors and the groupoid of fibre functors $\Rep(G) \rightarrow \Bun (X)$, and that it is functorial in $X/R$. In particular, $\iota_2^{-1} \circ \iota_1$ extends over $R$ if and only if the composite induced map $\omega_{\cP_1} \otimes_R K \rightiso \omega_{\cP_2} \otimes_R K$ on fibre functors extends over $R$, and this is the case if and only if the condition on lattices holds.
\end{proof}

\subsection{Special type case}

\subsubsection{}
Let $(T,h)$ be the Shimura datum defined by a torus split by a CM field, $E_T := E(T,h)$. Then $\Sh(T,h)/E_T$ is a product of algebraic field extensions, and we recall that its integral canonical model $\cS(T,h)$ is that obtained by taking integral closures. In particular, recall that if $T$ is unramified at all $p \nmid  N$, we get a unique integral model (also abusively written) $T/\Z[1/N]$ and letting $K^N = \prod_{p \nmid  N} T(\Zp)$ we get $\cS_{K^N}(T,h)$ ind-\'etale over $\cO_{E_T}[1/N]$ with a natural $\prod_{p|N}T(\Qp)$ action extending that on the generic fibre.

Recall that we also have $P_{K^N}(T,h) \rightarrow \Sh_{K^N}(T,h)$ the standard principal $T^c$-bundle, defined over $E_T$. Our aim here is to construct for it a canonical integral model. We assume henceforth that $T$ is unramified at $p \nmid  N$.

\subsubsection{}
\label{special canonical def}
Suppose $\cP/\cS_{K^N}(T,h)$ together with $\iota:\cP_{E_T} \rightiso P_{K^N}(T,h)$ is such a model. We can study the associated fibre functor 
$$\omega_{\cP}: \Rep_{\Z[1/N]}(T^c) \ni W \mapsto W \times^{T^c} \cP \in \vVec(\cS_{K^N}(T,h)).$$
The identification $\iota$ realises the image of such as lattices
$$\omega_{\cP}(W) \subset \omega_{P_{K^N}(T,h)}(W\otimes \Q)$$
in the ``de Rham sheaves'' defined by $P_{K^N}(T,h)$.

On the other hand, for each prime $q \nmid N$, and $v | q$ a place of $E_T$ we can let $K^{Nq} = \prod_{p \nmid Nq} T(\Zp)$ and it is immediate from the setup and class field theory that the pro-\'etale $T^c(\Z_q)$-cover $\Sh_{K^{Nq}}(T,h) \rightarrow \Sh_{K^N}$ gives rise to, for each representation $W_q$ of $T^c(\Z_q)$, a crystalline lisse $\Z_q$-sheaf $\omega_{et}(W_q)$ on $\Sh_{K^N}$. By (\ref{kismod}) we have associated to such a datum a canonical lattice $\D(\omega_{et}(W_q)) \subset D_{dR}(\omega_{et}(W_q) \otimes_{\Z_q} \Q_q)$.

Finally recall (\ref{dr identification}) which gives a natural identification
$$\theta: \omega_{P_{K^N}(T,h)}(W\otimes \Q) \otimes_{E_T} {E_{T,v}} \rightiso D_{dR}(\omega_{et}(W_q) \otimes_{\Z_q} \Q_q).$$

We say that the model $\cP$ is \emph{canonical} if for every $q \nmid  N$ and $W \in \Rep_{\Z[1/N]}(T)$ the two lattices $\omega_{\cP}(W) \otimes_{\cO_{E_T}[1/N]} \cO_{E_T,v}$ and $\D(\omega_{et}(W_q))$ constructed above are, under the map $\theta$, identified.

\begin{prop}
\label{torus canmod}
With $(T,h)$ as above, there exists a unique integral canonical model for $P_{K^N}(T,h)$.
\end{prop}
\begin{proof}
We first remark that uniqueness follows directly from (\ref{lattice uniqueness} (1)) and (\ref{lattice torsors}). 

For existence, let us first remark that we have the map of Shimura data $(T,h) \rightarrow (T^c,h^c)$ giving rise to a map of integral canonical models for the Shimura varieties $i:\cS_{K^N}(T,h) \rightarrow \cS_{K^{cN}}(T^c,h^c)$. Suppose $\cP^c$ is an integral canonical model for $P_{K^{cN}}(T^c,h^c)$. Then $i^*\cP^c$ is an integral canonical model for $P_{K^N}(T,h)$ because $i^*\omega_{et,(T^c,h^c)} = \omega_{et,(T,h)}$. Since any $h: \bS \rightarrow T^c_{\R}$ has weight defined over $\Q$, we are therefore reduced to the case where $(T,h)$ is a CM pair (i.e. where $T$ is split by a CM field and the weight of $h$ is defined over $\Q$).

Let $T \rightarrow GL(W)$ be a representation of $T/\Z[1/N]$. Since $(T,h)$ is a CM pair, $\Sh_{K^N}(T,h)$, a union of copies of $\Spec E$ for some field $E$, can be interpreted as a moduli space of CM motives with level structures. Let $P_{K^N}(T,h)^0 \rightarrow \Sh_{K^N}(T,h)^0 \rightiso \Spec E$ be any single component together with the restriction of $P_{K^N}(T,h)$ over it. There exists a finite extension $F/E$ and a CM abelian variety $A/F$ for such that $\omega_{P_{K^N}(T,h)^0}(W \otimes \Q) \otimes_E F \subset H^1_{dR}(A/F)^\otimes$ and $\omega_{et}(W) = H^1_{et}(A,\Zp)^\otimes \cap \omega_{et}(W) \otimes \Qp$ for all $p \nmid  N'$ some $N'$ divisible by $N$. Let $\cA/\cO_F$ be the Neron model of $A$, and notice that it is an abelian variety over $\cO_F[1/M]$ for some $M$ which we may assume is divisible by $N'$. Thus we construct an $\cO_E[1/M]$-lattice 
$$\Lambda' :=\omega_{P_{K^N}(T,h)^0}(W \otimes \Q) \cap H^1_{dR}(\cA/\cO_F[1/M])^\otimes \subset \omega_{P_{K^N}(T,h)^0}(W \otimes \Q).$$ 
The second part of (\ref{kismod}) assures us that at all $p \nmid  M$ this lattice agrees with that obtained via first taking the dual Tate module and applying Breuil-Kisin theory.

Moreover, using the construction of (\ref{kismod}) together with (\ref{lattice intersect}) for the finitely many $p$ which divide $M$ but do not divide $N$, we are able to extend $\Lambda'$ to a lattice $\Lambda$ over $\cO_F[1/N]$. This $\Lambda$ (applying the construction for all components) gives us the lattice in $\omega_{P_{K^N}(T,h)}(W \otimes \Q)$ we require to prove existence. Note that since we already have uniqueness, we may observe that the construction does not depend on the choices of $F$ and $A/F$.
\end{proof}

\subsection{Connections on $G$-bundles}
It will help to collect some basic facts about connections on $G$-bundles. Let $S$ be a scheme and $G/S$ a flat affine group scheme of finite type.

Suppose $X/S$ a scheme, and let $\Delta^2(1) = \Delta_{X/S}^2(1)$ be the first order neighbourhood of the diagonal in $X \times_S X$, $\delta: X \hookrightarrow \Delta(1)$ and $p_1,p_2:\Delta(1) \rightarrow X$ the two projection maps. We also will need $\Delta^3(1)$, the first order neighbourhood of the diagonal in $X \times_S X \times_S X$ and its three projections $p_{12},p_{23},p_{13}: \Delta^3(1) \rightarrow \Delta^2(1)$.

\subsubsection{}
Recall that if we have a vector bundle $\cV$ on $X$ a \emph{connection} on $\cV/X$ (relative to $S$) is given by an isomorphism
$$\nabla: p_1^*\cV \rightiso p_2^*\cV$$
such that $\delta^*\nabla = \id$ (under the canonical identification $(\delta \circ p_i)^*\cV = \id^*\cV \cong \cV$). Such a connection is said to be \emph{flat} if $$p_{13}^*(\nabla) = p_{23}^*(\nabla) \circ p_{12}^*(\nabla).$$
It is well known that these definitions are equivalent to the more usual definitions from differential geometry.

\subsubsection{}
Let $P \rightarrow X$ be a $G_X$-torsor. A \emph{connection} on $P$ is an isomorphism
$$\nabla: p_1^* P \rightiso p_2^* P$$ such that $\delta^*\nabla = \id_P$, and $\nabla$ is said to be flat, again if 
$$p_{13}^*(\nabla) = p_{23}^*(\nabla) \circ p_{12}^*(\nabla).$$

\begin{lem}
\label{torsor connection}
Let $X/S$ be a scheme, and assume $S$ is Dedekind. Let $P \rightarrow X$ be a $G$-bundle, with associated fibre functor $\omega_P$. To give the following pieces of data are equivalent.
\begin{enumerate}
\item A (flat) connection on $P \rightarrow X$.
\item For each representation $V \in \Rep_S(G)$ a (flat) connection on $\omega_{P}(V)$ in such a way that the isomorphisms
$$\omega_{P}(V \otimes W) \cong \omega_{P}(V) \otimes \omega_{P}(W)$$
and
$$\omega_{P}(V^\vee) \cong \omega_P(V)^\vee$$
are isomorphisms of bundles with connection.

\item Given a faithful representation $G \hookrightarrow GL_S(V)$ and tensors $s_{\alpha} \in V^\otimes$ such that $G = GL_{S,s_\alpha}(V)$, a (flat) connection on $\omega_{P}(V)$ with the property that each $\omega_{P}(s_\alpha \cO_S) \subset \omega_{P}(V^\otimes) = \omega_{P}(V)^\otimes$ receives the trivial connection.
\end{enumerate}
\end{lem}
\begin{proof}
The equivalence of (1) and (2) is formal given Broshi's Tannakian formalism over a Dedekind scheme \cite[1.2]{brosh}. Indeed the conditions in (2) are exactly those needed to say that the connections $\nabla_{\omega_P(V)}$ define an isomorphism of tensor functors $p_1^*\omega_P \rightiso p_2^*\omega_P$, which is the same as an isomorphism of torsors $p_1^*P \rightiso p_2^*P$, and one easily verifies that the conditions translate across.

It is obvious how to pass from (1) to (3). For going from (3) to (1), we note that $P$ can be canonically identified with the frame bundle $$P \cong \uIsom_{s_\alpha}(V \otimes \cO_X, P \times^G V),$$ (via $p \mapsto (v \mapsto (p,v))$). 
Let $\cV = \omega_P(V) = P \times^G V$ and $s_{\alpha,1} = (1,s_\alpha) \in \cV^\otimes$. If we are given a connection $\nabla:p_1^* \cV \rightiso p_2^* \cV$, which is trivial on each line containing $s_{\alpha,1}$, that is to say $\nabla(p_1^*s_{\alpha,1}) = p_2^*s_{\alpha,1}$, we obtain a well-defined isomorphism
$$\nabla_P: p_1^*P \ni \phi \mapsto \nabla \circ \phi \in p_2^*P,$$
giving the desired connection on $P$. Again it is easy to check that if $\nabla$ is flat then so is $\nabla_P$.
\end{proof}

\subsection{Definition of canonical models}
Let $(G,\fX)$ be a Shimura datum with reflex $E=E(G,\fX)$, and for what follows assume it admits an integral canonical model over $\cO_E[1/N]$ at any level hyperspecial away from $N$.

\subsubsection{}
We shall need the mild assumption that $Z(G)^\circ$ is split by a CM field, which we impose for the rest of the paper. This can be removed if the arguments relying on CM motives in proving (\ref{dr identification}) and (\ref{torus canmod}) can be replaced by arguments that work with greater generality.

\subsubsection{}

Recall (\ref{milne spb}) that the Shimura variety $\Sh(G,\fX)$ is equipped with a \emph{standard principal bundle} $P(G,\fX)$, which is a $G(\Af)$-equivariant $G^c$-torsor defined over $E$, with a flat connection $\nabla$ and a $G$-equivariant map $\gamma: P(G,\fX) \rightarrow Gr_\mu$. Suppose also $G$ is quasi-split and unramified away from $N$, fix $G=G_{\Z[1/N]}$ a reductive model, $K^N=\prod_{p \nmid N} G(\Zp)$ and for $\mu$ the Hodge cocharacter note that over an integral model one has the canonical inclusion $Gr_{\mu} \subset \mathcal{GR}_{\mu}.$

We also let $\omega_{et}: \Rep_{\Zp}(G^c) \rightarrow \Lisse_{\Zp}(\Sh_K(G,\fX))$ denote the standard \'etale $\Zp$-sheaves coming from the tower at $p$, and $\omega_{et,\eta}: \Rep_{\Qp}(G^c) \rightarrow \Lisse_{\Qp}(\Sh_K(G,\fX))$ the corresponding lisse $\Qp$-sheaves.

\subsubsection{}
\label{all the points}
We will say that a $(G,\fX)$ \emph{has enough crystalline points} if for every prime $\fp$ of $E$ with $\fp | p \nmid N$, and every $K_N \subset \prod_{q|N} G(\Q_q)$ compact open, $K=K_NK^N$, there is a dense subset $U_0$ of the special fibre of $\cS_{K_NK^N}/\fp$ with the property that for each $x_0 \in U_0$ there is a lift $\tilde{x}_0 \in \cS_{K_NK^N}(W(\kappa(x_0))$ such that $\tilde{x}_0[1/p]^*\omega_{et, \eta}$ takes values in crystalline representations of $\Gamma_{W(\kappa(x_0))[1/p]}$. 

We say $(G,\fX)$ has \emph{all the crystalline points} if for every $W(k)$-valued point $x \in \cS_{K_NK^N}$ for $k$ a finite field of characteristic $p \nmid N$, $x[1/p]^*\omega_{et, \eta}$ takes values in crystalline representations of $\Gamma_{W(k)[1/p]}$. Clearly in the present context where the base is smooth and so every $k$-point admits a $W(k)$-lift, if you have all the crystalline points you have enough crystalline points.

Moreover, we note that when $(G,\fX)$ is of Hodge type or special type it has all the crystalline points. In the Hodge type case, this is because the relevant Galois representations live inside $H^1_{et}(\cA_x[1/p],\Qp)^\otimes$, where $\cA_x$ is the fibre of the universal abelian scheme at $x$, and since $\cA_x/W(\kappa(x))$ is an abelian scheme, this representation is crystalline. In the special type case it follows from the explicit reciprocity law computation as in (\ref{special canonical def}). We shall later use the methods of \S4.6 to establish that $(G,\fX)$ of abelian type also has all the crystalline points.

In what follows, a \emph{crystalline point} is understood to be, for some $K_N$ and some finite field $k$ of characteristic $p \nmid N$, a point of $\cS_{K_NK^N}(W(k))$ such that $x[1/p]^*\omega_{et,\eta}$ takes values in crystalline representations, and any $p, K_N,k,K_0=W(k)[1/p]$ appearing will be understood to be part of this data.

\subsubsection{}
We will also say that to give $(G,\fX)$ a \emph{de Rham structure} is to give for every crystalline point $x \in \cS_{K_NK^N}(W(k))$,  an identification $$D_{dR} \circ x[1/p]^*\omega_{et, \eta} \rightiso x[1/p]^*\omega_{P_{K_NK^N}(G,\fX)}: \Rep_{\Qp}(G^c) \rightarrow \Bun_{K_0}.$$

In the case where $(G,\fX)$ is of abelian type (and we expect in general), there is a canonical de Rham structure. This follows using functoriality and the morphism $(G,\fX) \rightarrow (G^c,\fX^c)$ to reduce to the case where the weight is defined over $\Q$. Here we may use Milne's moduli interpretation \cite{milne9} in such a situation to interpret $x[1/p]$ as an abelian motive, and get the required identification by the same argument as (\ref{dr identification}). It is also straightforward to check that these canonical de Rham structures are compatible under morphisms of Shimura varieties induced by maps of Shimura data. In what follows, we will usually assume $(G,\fX)$ is of abelian type in which case we are always working with reference to the canonical de Rham structure.

\subsubsection{}
Let $(G,\fX)$ be a Shimura datum as above equipped with a de Rham structure. An \emph{integral canonical model} $\cP = \cP_{K^N}(G,\fX)$ for $P_{K^N}(G,\fX)$ is a $G^c$-torsor over the integral canonical model $\cS_{K^N}(G,\fX)$ for $\Sh_{K^N}(G,\fX)$ with an identification $$\iota: \cP_{K^N}(G,\fX) \otimes_{\cO_E[1/N]} E \rightiso P_{K^N}(G,\fX)$$ and equivariant $\prod_{p|N}G(\Qp)$-action such that for any crystalline point $x \in \cS_{K_NK^N}(W(k))$ we have the \emph{lattice property} given below.

Consider $$\omega_{dR,x}: D_{dR} \circ x[1/p]^*\omega_{et,\eta}: \Rep_{\Qp}(G^c) \rightarrow \Bun_{K_0}.$$
This functor comes with two canonical lattices. First, there is the lattice given by $\omega_{\cP_{K_NK^N}}$ coming via $\iota$ and the de Rham structure. Second, there is the lattice $\D \circ x[1/p]^*\omega_{et}$ coming from the theory (\ref{kismod})  of $\fS$-modules. The lattice property requires that these lattices are equal. Note that since everything is Hecke equivariant, it suffices to check the lattice property at infinite level, an observation of which we shall make liberal use.

We will also insist that for $\cP_{K^N}$ to qualify as a canonical model the connection $\nabla$ and the $\mu$-filtration $\gamma$ extend to $\cP$, although we shall see that in the abelian type situation these additional properties are automatic given the condition on crystalline points.

\subsubsection{}
We remark that this definition is compatible with (\ref{special canonical def}). We also remark that one can make definitions of integral canonical models defined over any unramified-away-from-$N$ extension $F/E$, and for any such $F$ containing the abelian extension over which $\Sh_{K^N}(G,\fX)$ splits into its geometrically connected components we may also define integral canonical models for $P_{K^N}(G,\fX)_{F}^+$. These definitions are made in exactly the same way and such models enjoy the properties we are about to note for the same reasons.

\begin{lem}
\label{uniqueness}
Suppose $(G,\fX)$ has enough crystalline points. An integral canonical model $(\cP,\iota)$, if it exists, is unique up to canonical isomorphism.

This is true without assuming $\nabla$ or $\gamma$ extend.
\end{lem}
\begin{proof}
Suppose we have two $(\cP,\iota), (\cP',\iota')$ integral canonical models for $P_{K^N}(G,\fX)$. We will aim to show that $\iota'^{-1} \circ \iota: \cP \otimes E \rightiso \cP' \otimes E$ extends over $\cO_E[1/N]$. By (\ref{lattice torsors}) it suffices to check that for any representation of $G_{\Z[1/N]}^c$ on a finite free module $V$ that the composite $(\iota'^{-1} \circ \iota)_*: \omega_{\cP}(V) \otimes E \rightiso \omega_{\cP'}(V) \otimes E$ identifies the $\cO_E[1/N]$-lattice $\omega_{\cP}(V)$ with $\omega_{\cP'}(V)$.

By Hecke equivariance we may reduce to checking this at every finite level $K_NK^N$. Noting that $\cS_{K_NK^N}$ is smooth, we may invoke our lattice lemma (\ref{lattice uniqueness}) to reduce the statement first (\ref{lattice uniqueness}(1)) to checking equality of lattices over the formal completion of $\cS_{K_NK^N}$ at each (maximal) prime $\fp$ of $\cO_E[1/N]$, and then (\ref{lattice uniqueness} (2)) to checking equality of lattices on lifts of a dense subset of the special fibre at $\fp$. But since $(G,\fX)$ has enough crystalline points, and by the lattice condition we have the necessary equality at these points, we get the desired identification.

Note that the connection and $\mu$-filtration are extended from the generic fibre and that they extend is a property, so they have no impact on the uniqueness statement.
\end{proof}

The following functoriality properties can now be read off.

\begin{prop}
\label{func}
Let $f: (G_1,\fX_1) \rightarrow (G_2,\fX_2)$ be a morphism of Shimura data with enough crystalline points and compatible de Rham structures induced by a map $G_{1,\Z[1/N]} \rightarrow G_{2,\Z[1/N]}$ of reductive groups over $\Z[1/N]$, and $K^N_i = \prod_{p \nmid  N}G_i(\Zp)$ and $(\cP_i,\iota_i)$ an integral canonical model for $P_{K^N_i}(G_i,\fX_i)$, $i=1,2.$
\begin{enumerate}
\item We can canonically identify
$$\cP_1 \times^{G_1^c} G_2^c \rightiso f^*\cP_2.$$
\item The induced diagram 
$$
\begin{CD}
    \cP_1 @>\gamma_1 >> \mathcal{GR}_{\mu_1} \\
@VVV        @VVV\\
    \cP_2 @>\gamma_2 >>  \mathcal{GR}_{\mu_2}
\end{CD}
$$
commutes.
\end{enumerate}
\end{prop}
\begin{proof}
Let $E=E(G_1,\fX_1)$. Note that by (\ref{spb functorial}) we have a natural identification
$$\theta: P_{K^N_1}(G_1,\fX_1) \times^{G_1^c} G_2^c \rightiso f^*P_{K^N_2}(G_2,\fX_2)$$
and so as in the previous lemma it will suffice to check that 
$$f^*(\iota_2^{-1})\circ \theta \circ \iota_1: (\cP_1 \times^{G_1^c} G_2^c) \otimes_{\cO_E[1/N]} E \rightiso (f^*\cP_2) \otimes_{\cO_E[1/N]} E$$
exchanges the natural lattices when taken on the level of fibre functors $$\omega: \Rep_{\Z[1/N]} G_2^c \rightarrow \vVec(\cS_{K_1^N}(G_1,\fX_1)).$$ As in the previous lemma, we may reduce to working at finite level and checking equality at crystalline points. But this is then immediate from the lattice condition, the compatibility of the \'etale fibre functors on both Shimura varieties, together with the (consequent) observation that the image of a crystalline point is always a crystalline point. This completes the proof of (1).
Now (2) follows directly from the fact that after taking $\otimes_{\cO_E[1/N]} E$ the diagram commutes by (\ref{spb functorial}). Note that the right hand vertical map is that given by (\ref{gr functoriality}).
\end{proof}

\begin{lem}
\label{galois descent}
Suppose $F/E=E(G,\fX)$ a Galois extension unramified away from $N$, and $(\cP'/\cO_F[1/N],\iota')$ an integral canonical model for $P_{K^N}(G,\fX)_F$.

Then the descent data $$\Gal(F/E) \ni \sigma \mapsto \theta_\sigma: \sigma^*P_{K^N}(G,\fX)_F \rightiso P_{K^N}(G,\fX)_F$$
extend to $\cP'$, and the pair $(\cP,\iota)$ obtained by \'etale descent is an integral canonical model for $P_{K^N}(G,\fX)$ over $\cO_E[1/N]$.
\end{lem}
\begin{proof}
Take $\sigma \in \Gal(F/E)$. We claim that $\sigma^*\cP'_{K^N}$ together with the composite
$$\sigma^*\cP'_{K^N} \otimes F \map{\sigma^*\iota} \sigma^*P_{K^N}(G,\fX)_F \map{\sigma} P_{K^N}(G,\fX)_F$$
is an integral canonical model. From this claim it is immediate that the descent data extend and in its turn $\cP'_{K^N}$ descends to an integral canonical model over $\cO_E[1/N]$, because being a crystalline point and the lattices coming from $\fS$-modules are stable under finite unramified base change.

We may pull back the Hecke action, so it remains to check the lattice condition on crystalline points at finite level. Let $x \in \cS_{K_NK^N}(G,\fX)(W(\kappa))$ be a crystalline point and $v$ the corresponding place of $F$, and $\sigma^*(x)$ its pullback along $\sigma: \cS_{K_NK^N} \rightiso \cS_{K_NK^N}.$ Let $\omega_{et,x}$ and $\omega_{et,\sigma^*(x)}$ be the $\Zp$-linear \'etale fibre functors coming from the Shimura variety at infinite level at $p$ and restricting to the fibres at $x[1/p]$ and $\sigma^*(x)[1/p]$ respectively, viewing both as taking values in $\Gamma_{K_0}$-representations. Note that by construction (since $\Sh(G,\fX)$ is defined over $E$) there is an isomorphism $\omega_{et, \sigma^*(x)} \rightiso \omega_{et,x}$ covering $\sigma: \sigma^*{x} \rightiso x.$ In particular, $\sigma^*(x)$ is also crystalline.

For $V \in \Rep_{\Zp}(G^c)$ we must check an equality of two lattices in $D_{dR}(\omega_{et,x}(V)[1/p])$. The first is the usual $\D (\omega_{et,x}(V))$ coming straight from $\fS$-modules, the second obtained as the composite
$$\D(\omega_{et,\sigma^*(x)}(V)) \subset D_{dR}(\omega_{et,\sigma^*(x)}(V)[1/p]) \rightiso D_{dR}(\omega_{et,x}(V)[1/p]).$$
Having thus spelt it out, we see that it follows immediately from functoriality of the $\D$ construction. Moreover, the $G^c$-action, connection and filtration are all defined over $E$ on the generic fibre and it can be checked they extend to $\cO_E[1/N]$ \'etale locally, so the fact they do over $F$ immediately gives them all.

\end{proof}

\subsection{Hodge type case}
We turn our attention to existence.
\begin{thm}
\label{main hodge}
Let $(G,\fX)$ be a Shimura variety of Hodge type with $G$ unramified away from $N$, $K^N = \prod_{p\nmid N}G(\Zp)$. Then $P_{K^N}(G,\fX)$ admits an integral canonical model.
\end{thm}
We let $E=E(G,\fX)$ throughout this paragraph, and unless otherwise specified all Shimura varieties $\Sh_K,\cS_K$ or bundles $P_K$ or $\cP_K$ will be understood to be those attached to the Shimura datum $(G,\fX)$. We also note that $(G,\fX)$ being of Hodge type implies $G=G^c$.

\subsubsection{}
First some more algebraic preliminaries. Let $R$ be a ring, $M$ a finite free $R$-module, and $M^\otimes$ the direct sum of all $R$-modules formed from $M$ by the operations of taking duals, tensor products, symmetric powers and exterior powers. If $s_\alpha \in M^\otimes$ is a collection of tensors, we say that they \emph{define} a subgroup $G \subset GL(M)$ if it is precisely the group which acts trivially on all the $s_\alpha$.

We need the following version of \cite[1.3.2]{kis2}, whose proof is basically identical.

\begin{prop}
Let $R$ be a PID with field of fractions $K$, $M$ a finite free $R$-module, and $G \subset GL(M)$ a closed subgroup, flat over $R$ with reductive generic fibre. Then it is defined by a finite collection of tensors $s_\alpha$.
\end{prop}
\begin{proof}
Exactly as in \cite[1.3.2]{kis2} we can reduce to showing that it suffices to find tensors defining $G$ in some representation of $GL(M)$ on a finite projective $R$-module, and we take $I \subset \cO_{GL(M)}$ the ideal of $G$ in the Hopf algebra of $GL(M)$, which has scheme-theoretic stabiliser precisely $G$.

By \cite[3.3]{wat} there is a finite dimensional $GL(M)_K$-stable subspace $W_\eta \subset \cO_{GL(M)} \otimes K$ containing a finite set of generators for $I$ as an ideal. Let $W = W_{\eta} \cap \cO_{GL(M)}$, which is visibly $GL(M)$-stable, and it is finitely generated because $\cO_{GL(M)}$ is a free $R$-module and letting $e_1,...,e_n$ be a basis for $W_{\eta}$, writing them as elements of $\cO_{GL(M)} \otimes K$ we see they lie in $\tilde{W} \otimes K$ for some $\tilde{W} \subset \cO_{GL(M)}$ a finite free $R$-submodule, and such that $W \subset \tilde{W}$. But $R$ is a PID, so we deduce $W$ is finite free.

Now we see that $G$ is the stabiliser of $W \cap I \subset W$, and the argument can be concluded exactly as in \cite[1.3.2]{kis2}.

\end{proof}

\subsubsection{}
Now recall (\ref{hodge group construction}) that for $(G,\fX)$ our Shimura datum of Hodge type with $G$ unramified away from $\Z[1/N]$ (taking $G_{\Z[1/N]} = \cG$ from the discussion in (\ref{hodge group construction})) we may choose a symplectic embedding $i: (G,\fX) \hookrightarrow (GSp_{2g},S^\pm)$ extending to a closed immersion
$$i:G_{\Z[1/N]} \hookrightarrow GL(V_{\Z[1/N]}).$$
Henceforth view $V$ as a $\Z[1/N]$-module, and $G$ as a reductive group over $\Z[1/N]$. By the above proposition we can find $s_\alpha \in V^\otimes$ such that $G = \Aut_{s_\alpha}(V)$.

\subsubsection{}

The map $i$ defines (pulling back the universal abelian variety) an abelian variety $\pi:A_{K^N} \rightarrow \cS_{K^N}$, and we can study its sheaf of relative de Rham cohomology $\cV = \cH^1_{dR}(A_{K^N}/\cS_{K^N})$ with Gauss-Manin connection $\nabla$. 

By the absolute Hodge cycles argument of \cite[2.2.2]{kis2} we obtain $\nabla$-horizontal $\prod_{p|N}G(\Qp)$-invariant tensors $s_{\alpha,dR} \in \cV^\otimes \otimes_{\Z[1/N]} \Q$, and by \cite[2.3.9]{kis2} for any $v \nmid  N$ a finite place of $E$ these sections extend over $\cO_{E,(v)}$, which is enough to conclude they in fact lie in $\cV^\otimes.$

Let us therefore define the sheaf on $(\cS_{K^N})_{fppf}$
$$\cP_{K^N} := \uIsom_{s_\alpha}(V, \cV)$$
whose sections over $u:U \rightarrow \cS_{K^N}$ are those trivialisations $\theta: V \otimes_{\Z[1/N]} \cO_U \rightiso u^*\cV$ such that (after applying $(-)^\otimes$ to both sides) each $s_\alpha$ is identified with $s_{\alpha,dR}$. Note that since $\prod_{p|N}G(\Qp)$ acts equivariantly on $\cV$ and fixes $s_{\alpha,dR}$ we may freely pass between infinite and finite level.

\begin{prop}
The sheaf $\cP_{K^N}$ is a $G$-torsor over $\cS_{K^N}$.
\end{prop}
\begin{proof}
It obviously suffices to check at finite level $K=K_NK^N$ for some sufficiently small $K_N$. We claim it also suffices to check in the formal neighbourhood of any closed point $x \in \cS_K$. 

First we may directly verify that $\cP_K$ is representable because passing to a Zariski cover making $\cV$ free, and then choosing arbitrary identifications $V \cong \cV$, $\cP_K$ can be constructed as a closed (hence affine) subvariety of $GL(V)$. Now suppose we know it is a $G$-torsor in a formal neighbourhood of every closed point. Then since $G$ is faithfully flat, $\cP_K$ is faithfully flat over the formal neighbourhood of every closed point, which is enough to prove that $\cP_K/\cS_K$ is faithfully flat. It remains to show that the natural map
$$\cP_K \times_{\cS_K} G_{\cS_K} \ni (p,g) \mapsto (p,pg) \in \cP_K \times_{\cS_K} \cP_K$$
is an isomorphism, but this also is the case iff it is the case over a formal neighbourhood of each closed point, where again it follows from our assumption.

Thus we are reduced to the study of $\hat{P}_x := \cP_{K}|_{\hat{\cS_{K,x}}}$ for closed points $x \in \cS_K$. If $\Char(x) = 0$, the Betti de-Rham comparison theorem gives an identification $V_{\hat{\cS_{K,x}}} \otimes \C \rightiso \hat{\cV_{x}}\otimes \C$, giving a fpqc-local section. This suffices because it shows $\hat{P}_x$ is smooth by fpqc descent, and so $\hat{P}_x$ admits a section \'etale locally, which is enough to see it is a $G$-torsor.

If $\Char(x)=p$, we may follow the proof of \cite[2.3.5]{kis2} to study $N_x := \cO_{\hat{\cS_{K,x}}}$ together with the $p$-divisible group $\cG = A_K|_{\hat{\cS_{K,x}}}[p^\infty]$ over the formally smooth ring $N_x$.

First note that $x$ determines a place $v$ of $E$ above $p$ and since $E$ is unramified at $p$ we may identify $\cO_{E,v} = W(\kappa(x)) =: W$, over which $N_x$ is canonically an algebra. The crystalline-de Rham comparison gives an identification $\cV_{N_x} = \D(\cG)(N_x)$. We let $\cG_0$ denote the restriction of $\cG$ to $x$.

Taking $\tilde{x}: N_x \rightarrow W$ a lift, we may follow the argument of \cite[2.3.5]{kis2} to obtain from the Tate module of $\tilde{x}^*\cG$ tensors $s_{\alpha,0} \in \D(\cG_0)(W)^\otimes$ defining a reductive subgroup $G(\tilde{x}) \subset GL(\D(\cG_0)(W))$, which gives rise to an explicit versal deformation ring $R_{G(\tilde{x})}$ of $G(\tilde{x})$-adapted deformations of $\cG_0$. This then has the property that we may identify $R_{G(\tilde{x})} \rightiso N_x$ in such a way that induces an identification of $\cG$ with the explicit versal deformation (let's also call it $\cG$) one constructs over $R_{G(\tilde{x})}$, so in particular we have
$$\cV_{N_x} = \D(\cG)(R_{G(\tilde{x})}).$$
As in \cite[2.3.9]{kis2} one has by an explicit construction lifts $\tilde{s}_{\alpha,0}$ of $s_{\alpha,0}$ to $\D(\cG)(R_{G(\tilde{x})})^\otimes$ which are identified with $s_{\alpha,dR} \in \cV_{N_x}^\otimes$.
This explicit construction comes about in \cite[\S1.5]{kis2} by defining $\D(\cG)(R_{G(\tilde{x})}) = \D(\cG_0)(W) \otimes_W R_{G(\tilde{x})}$ as a module and taking the lift $\tilde{s}_{\alpha,0} = s_{\alpha,0} \otimes 1.$
Finally let us note that by \cite[1.4.3]{kis2} (since $G$ is connected) there exists a $W$-linear isomorphism
$$V_{W} \rightiso \D(\cG_0)(W)$$
identifying $s_\alpha$ with $s_{\alpha,0}$. In particular this identifies $G_W \rightiso G(\tilde{x})$ uniquely up to inner automorphisms and the composite

$$V_W \otimes_W R_{G(\tilde{x})} \rightiso \D(\cG_0)(W) \otimes_W R_{G(\tilde{x})} \rightiso \D(\cG)(R_{G(\tilde{x})}) = \cV_{N_x}$$
gives a section of $\hat{P}_x$ as required.
\end{proof}

\subsubsection{}
We would now like to show it is the desired canonical model for $P_{K^N}(G,\fX)$. Firstly we must check that the additional structures extend.

We have already noted that the $\prod_{p|N} G(\Qp)$-action extends. By (\ref{torsor connection}), the connection $\nabla$ on $P_{K^N}$ extends to $\cP_{K^N}$ because the Gauss-Manin connection extends to (i.e. can be constructed directly on) $\cV = \cH^1_{dR}(\cA_{K^N}/\cS_{K^N})$ and by the construction of $s_{\alpha,dR}$ in \cite[2.2]{kis2} it is clear they are parallel sections. It is also immediate the $\mu$-filtration extends, by (\ref{filtration extension}) and that the Hodge filtration is naturally defined on the integral de Rham cohomology $\cV$.

Finally, we need to verify that it is indeed a canonical model, and as usual we check the lattice condition at a crystalline point $x$. But in our situation this is very straightforward: we just use the facts that, taking $\cA_x$ to be the pullback of the universal abelian scheme to $x$, that (\ref{kismod}) $\D(H^1_{et}(\cA_x,\Zp)) = H^1_{dR}(\cA_x/W(k))$ and that by compatibiity of $\D$ with the $p$-adic comparison theorems, $\D(x^*s_{\alpha,et}) = x^*s_{\alpha,dR}$.

\subsubsection{}
\label{ad extension}
For making the transition from Hodge to abelian type, we will need to extend the equivariant $G^c \times \prod_{p|N}G(\Qp)$-action to an equivariant action of the $\Z^N = \Z[1/N]$-group scheme
$$\cA^N_P(G) = (G^c \times \frac{\prod_{p|N}G(\Qp)}{\overline{Z(\Z^N)}}) *_{G(\Z^N)_+/Z(\Z^N)} G^{ad}(\Z^N)^+$$
which sits in an exact sequence
$$1 \rightarrow G^c \rightarrow \cA_P^N(G) \rightarrow \cA^N(G) \rightarrow 1.$$
To do this it suffices to give the $G^{ad}(\Z^N)^+$-action and then check it is compatible, for which we imitate the approach taken in \cite[3.2]{kis2} and \cite[4.5]{kispap}, with our modifications from (\ref{adjoint action}) and further slight modifications.

To be precise, recall that a point $x \in \cP_{K^N}(T)$ gives the data of an quadruple $(A,\lambda,\epsilon_{et},\epsilon_{dR})$
where $A/T$ is an abelian scheme up to an isogeny whose degree is supported on $N$, $\lambda$ a weak polarisation of $A$, $\epsilon_{et} \in \Gamma(T, \uIsom_{s_{\alpha}}(V_{\prod_{p|N}\Qp}, \prod_{p|N}\hat{V}_p(A)))$, and $\epsilon_{dR}$ a section of $\uIsom_{s_{\alpha}}(V_T, \cV_T)$.

Given $\gamma \in G^{ad}(\Z^N)^+$, we may form the $Z^{der}$-torsor $\cP = \{g \in G^{der}|\pi(g)=\gamma\}$. We may take $F/\Q$ Galois such that $\cP(\cO_F[1/N])$ is nonempty, and take $\tilde{\gamma} \in \cP(\cO_F[1/N])$. This can be used to define an action just as in \cite[3.2]{kis2}. Indeed we can define $$A^\cP(T) = (A(T)\otimes_{\Z[1/N]} \cO_\cP)^{Z^{der}}$$
and specialise the map
$$A^\cP \otimes_{\Z[1/N]} \cO_{\cP} \rightiso A \otimes_{\Z[1/N]} \cO_{\cP}$$
at $\tilde{\gamma}$ to obtain
$$\iota_{\tilde{\gamma}}: A^{\cP} \otimes_{\Z[1/N]} \cO_F[1/N] \rightiso A \otimes_{\Z[1/N]} \cO_F[1/N].$$

These allow one to give the action of $\gamma$ as taking $(A,\lambda,\epsilon_{et},\epsilon_{dR}) \mapsto (A^\cP, \lambda^\cP, \epsilon_{et}^\cP, \epsilon_{dR}^\cP)$, where all but the last of these are as defined by Kisin. For the action on the de Rham component, we need the following modification of \cite[3.2.5]{kis2}. 

\begin{lem}
With notation as above, the composite
$$V_T \otimes_{\Z^N} \cO_F^N \map{\tilde{\gamma}^{-1}} V_T \otimes_{\Z^N} \cO_F^N \map{\epsilon_{dR}} H^1_{dR}(A/T) \otimes_{\Z^N} \cO_F^N \map{\iota_{\tilde{\gamma}}^{-1}} H^1_{dR}(A^\cP/T) \otimes_{\Z^N} \cO_F^N$$
is $Gal(F/\Q)$-invariant, hence induces a section
$$\epsilon_{dR}^\cP:V_T \rightiso H^1_{dR}(A^\cP/T)$$
which takes $s_{\alpha}$ to $s_{\alpha,dR}$.
\end{lem}
\begin{proof}
The cocycle computation argument of \cite[3.2.5]{kis2} goes through unchanged. For the claim about tensors, the only non-obvious point is that $\iota_{\tilde{\gamma}}^{-1}$ preserves the tensors, which one checks directly by working over $\C$, using the comparison with Betti cohomology as in \cite[3.2.6]{kis2}.
\end{proof}
It is also clear that this action is compatible with the $G^c \times \prod_{p|N} G(\Qp)/\overline{Z(\Z^N)}$-action in the sense that if we are given $g \in G^c(\Z^N)_+$ we may take $F=\Q$, $\tilde{\gamma}=\gamma=g^{-1}$ and post-compose the whole quadruple by the quasi-isogeny $\iota_{\gamma}$ to identify 
$$(A,\lambda,\epsilon_{et} \circ g, \epsilon_{dR} \circ g) = (A^\cP, \lambda^\cP, \epsilon_{et}^\cP, \epsilon_{dR}^{\cP}).$$

\subsection{Some distinguished Shimura data}
In this section we digress a little. To get from the Hodge type to the abelian type case we would like to follow Deligne \cite{del2} and pass via connected components, but unfortunately a straightforward reduction of $\cP$ to the derived group is not possible to carry out over $\cO_E[1/N]$ (roughly speaking because carrying out such a construction requires ``trivialising the $G^{ab}$ part of the motivic structure'', which can only be done canonically over $\C$ using Betti-de Rham comparison). We address this by constructing some distinguished Shimura data $\cB(G^{der},\fX^+,E)$ which are in some sense initial among Shimura data whose connected Shimura datum is $(G^{der},\fX^+)$ and whose reflex field is contained in $E$. Pulling back our torsors from the Hodge type case we may then reduce along $\cB \rightarrow G$ while leaving all motivic structures intact.

\subsubsection{}
Let $(G,\fX)$ be a Shimura datum, and $E$ its reflex field. Our goal is to construct a new Shimura datum $(\cB,\fX_\cB) \rightarrow (G,\fX)$ with $\cB^{der}=G^{der}$ and reflex field $E$ and show it depends only on these data and $\fX^+$ and not on $G$.

Since $E$ is the reflex field, we obtain a well-defined $$\mu_E:\G_{m,E} \rightarrow G^{ab}_E$$ taking the `determinant' (composite with $G \rightarrow G^{ab}$) of any Hodge cocharacter. Taking Weil restriction along $E/\Q$ and the norm map, we get a composite
$$r: E^* := \Res_{E/\Q}\mathbb{G}_{m,E} \rightarrow \Res_{E/\Q} G^{ab}_E \rightarrow G^{ab}.$$

Define $\cB = G \times_{G^{ab}} E^*$, where the first map is the natural $\delta:G \rightarrow G^{ab}$. It is straightforward to check the following.

\begin{lem}
The scheme $\cB$ is a reductive group scheme over $\Q$, sitting in a short exact sequence of $\Q$-group schemes
$$0 \rightarrow G^{der} \rightarrow \cB \rightarrow E^* \rightarrow 0.$$
\end{lem}

Suppose now that $(G,\fX)$ induces a connected Shimura datum $(G^{der},\fX^+)$ with reflex field $E(G^{der},\fX^+):=E(G^{ad},\fX^{ad}) \subset E$. We could make the above construction varying over all Shimura data with connected datum $(G^{der},\fX^+)$ and reflex field contained in $E$. The following is an adaptation of the argument in \cite[2.5]{del2}.

\begin{prop}
The extension $$0 \rightarrow G^{der} \rightarrow \cB \rightarrow E^* \rightarrow 0$$
depends only on the connected Shimura datum $(G^{der},\fX^+)$ and the field $E$.
\end{prop}
\begin{proof}
Let $(G,\fX)$, $(G',\fX')$ be two Shimura data whose reflex fields are both contained in $E$ and which give rise to the same connected Shimura datum. By \cite[2.5.6]{del2} we can find a third such Shimura datum $(G'',\fX'')$ which admits maps $$(G, \fX) \leftarrow (G'',\fX'') \rightarrow (G',\fX').$$

It will therefore suffice to show that for any morphism $\alpha: (G,\fX,E) \rightarrow (G_2,\fX_2,E_2)$ of Shimura data with a field of definition (where $E \rightarrow E_2$ represents an inclusion $E_2 \subset E$) we have an induced natural morphism of extensions
$$
\begin{CD}
    0 @>>> G^{der} @>>> \cB @>>> E^* @>>>  0\\
@.      @VVV@VVV@VVV @. \\
      0 @>>> G^{der}_2 @>>> \cB_2 @>>> E^*_2 @>>>  0 
\end{CD}
$$
where the outer maps are the natural ones induced by $\alpha: G \rightarrow G_2$ and $N_{E/E_2}:E^* \rightarrow E_2^*$. Since in the above special case both these maps are isomorphisms, we will get a canonical identification
$$\cB \leftiso \cB'' \rightiso \cB'$$
of extensions, proving the proposition.

Let us verify the claim. Recall that to give a map $\cB \rightarrow \cB_2$ is to give maps $f: \cB \rightarrow G_2$ and $k: \cB \rightarrow E_2^*$ that agree when extended to $G_2^{ab}$. Write $N\mu:E^*\rightarrow G^{ab}$ and $\delta:G \rightarrow G^{ab}$ and $N\mu_2, \delta_2$ for the analogues for $G_2$. Let us take $f$ to be the natural composite
$$f: \cB \map{pr_G} G \rightarrow G_2$$
and $k$ to be
$$k: \cB \map{pr_{E^*}} E^* \map{N_{E/E_2}} E_2^*.$$
We must check that whenever $\delta(g) = N\mu(e)$, we have
$$\delta_2(\alpha(g)) = N\mu_2(N_{E/E_2}(e)).$$

But by functoriality of $(-)^{ab}$
$$\delta_2(\alpha(g)) = \alpha^{ab}(\delta(g)) = \alpha^{ab}(N\mu(e)) = N\mu_2(N_{E/E_2}(e))$$
where the final equality holds because since $\alpha$ is a morphism of Shimura data we have in particular that $\alpha^{ab} \mu = \mu_2:\G_{m,\C} \rightarrow G_{2,\C}^{ab}$, from which it follows that $\alpha^{ab} \circ N\mu = N\mu_2 \circ N_{E/E_2}$ by functoriality of the norm map between tori. Thus we have a map $\cB \rightarrow \cB_2$, and it is clear from its definition that it fits into the diagram described.
\end{proof}

\subsubsection{}
Our next task is to define a Shimura datum. Take any $h_G \in \fX^+$, and we define a canonical map
$$h_E: \C^* \rightarrow E^*(\R)$$
as follows. Let $\tau: E\hookrightarrow \C$ be the canonical inclusion of the reflex field $E$ into $\C$, and write $E^*(\R) = E_\tau^* \times \prod_{\tau' \not=\tau} E_{\tau'}^*$. If $\tau$ is real, set $$h_E(z) = (z\bar{z};1,1,...,1).$$
If $\tau$ is complex,
$$h_E(z) = (z;1,...,1)$$
(i.e. the entries away from the place $\tau$ are all trivial in both cases).

\begin{prop}
These give a map $h_G \times h_E: \mathbb{S} \rightarrow \cB_\R$ which defines a Shimura datum with a natural map $(\cB,\fX_{\cB}) \rightarrow (G,\fX)$. Moreover, it is independent of the choice of $h_G$, has reflex field $E$, and only depends on $E$ and the connected Shimura datum $(G^{der},\fX^+)$ up to canonical isomorphism.
\end{prop}
\begin{proof}
We must check it is defined, which amounts to proving that $r_{\R}(h_E(z)) = \delta(h_G(z))$. First we note that
$$\delta(h_G(z)) = \mu_E(z)\mu_E(\bar{z}) = \mu_E(z) \overline{\mu_E(z)}$$
where we view $\mu_E$ as a character $\G_{m,\C} \rightarrow G^{ab}_{\C}$ using the canonical embedding $\tau$ of $E$.

On the other hand recall that $$r_{\R}=(N_{E/\Q} \circ \mu_E)_{\R}:E^*_{\R} \rightarrow G^{ab}_{E\otimes_\Q \R} \rightarrow G^{ab}_\R.$$
If $\tau$ is real then we have
$$(z\bar{z},1,1,...,1) \mapsto (\mu_E(z\bar{z}),1,...,1) \mapsto \mu_E(z\bar{z}) = \mu_E(z) \overline{\mu_E(z)}$$
and if $\tau$ is complex then
$$(z,1,1,1,1) \mapsto (\mu_E(z),1,...,1) \mapsto \mu_E(z) \overline{\mu_E(z)},$$
where again all occurrences of $\mu_E$ are viewed over $\C$ via $\tau$.

In particular we have checked the necessary equality.

It defines a Shimura datum because after projecting along $\cB \rightarrow G^{ad}$, $h_G \times h_E$ agrees with $h_G$. This is independent of the choice of $h_G$ because in general $\fX$ is a union of copies of $\fX^+$ so every other choice for $h_G \in \fX^+$ must be obtained. For the statement about the reflex field, we note that it is immediate from the general fact that $E(\cB,\fX) = E(\cB^{ad},\fX_{\cB}^{ad})E(\cB^{ab},h) = E$.

To show it only depends on $(G^{der},\fX^+)$ and $E$ we just return to the argument above for $\cB$ being independent. Given $(G,\fX), (G',\fX')$ Shimura data with reflex fields contained in $E$ and the same connected Shimura datum, we can again apply \cite[2.5.6]{del2} and find $(G'',\fX'')$ mapping to both. Since the construction of $(\cB,\fX_{\cB})$ is functorial in $(G,\fX)$, choosing $h \in \fX^+$ immediately gives a commuting diagram

$$
\begin{CD}
    \bS @= \bS @= \bS\\
@VVV  @VVV     @VVV\\
    \cB_{\R} @<\cong << \cB''_{\R}  @>\cong >> \cB'_{\R}.
\end{CD}
$$
From this it is clear that $(\cB,\fX_{\cB})$ depends only on $(G^{der},\fX)$ and $E$.
\end{proof}

We would like to check the obvious functorialities for the pair $(\cB,\fX_{\cB})$.

\begin{lem}
\label{distconst functorial}
The above construction is functorial in the following senses.
\begin{enumerate}
\item Let $u:(G_1,\fX_1^+) \rightarrow (G_2,\fX_2^+)$ be a map of connected Shimura data: that is, a central isogeny $G_1 \rightarrow G_2$ of semisimple groups such that $\fX_1^{+,ad} = \fX_2^{+,ad}$, and $E \supset E(G_1,\fX_1^+)$. Let $(\cB_1,\fX_1),(\cB_2,\fX_2)$ be the Shimura data associated to the triples $(G_1,\fX_1^+,E)$ and $(G_2,\fX_2^+,E)$ by the above procedure. Then there is a canonical morphism
$$u_*: (\cB_1,\fX_1) \rightarrow (\cB_2,\fX_2)$$
of Shimura data also induced by a natural central isogeny $\cB_1 \rightarrow \cB_2$.
\item Let $(G^{der},\fX^+)$ be a connected Shimura datum, $E(G^{der},\fX^+) \subset F \subset E$ two fields, and $(\cB_F,\fX_F),(\cB_E,\fX_E)$ the associated data. Then the norm map induces a canonical morphism of Shimura data
$$N_{E/F}: (\cB_E,\fX_E) \rightarrow (\cB_F,\fX_F).$$

\end{enumerate}
\end{lem}

\begin{proof}
For (1), let $\Delta_1 \subset G_1 \subset \cB_1$ be the kernel of $G_1 \rightarrow G_2$ viewed as a subgroup of the centre of $\cB_1$. Then $(\cB_1/\Delta_1, \fX_1/\Delta_1)$ is a Shimura datum with connected Shimura datum $(G_2,\fX_2^+)$ and reflex field contained in $E$. It therefore receives a natural map $(\cB_2,\fX_2) \rightarrow (\cB_1/\Delta_1, \fX_1/\Delta_1)$ which we claim is an isomorphism. It will suffice to check $\cB_2 \rightarrow \cB_1/\Delta_1$ is an isomorphism, which is immediate because it is a pullback of the identity map on $E^*$. Thus we obtain the canonical map
$$(\cB_1,\fX_1) \rightarrow (\cB_1/\Delta_1, \fX_1/\Delta_1) \leftiso (\cB_2,\fX_2)$$
required.

For (2), for any Shimura datum $(G,\fX)$ with connected Shimura datum $(G^{der},\fX^+)$ and reflex field contained in $F$, one obviously has a canonical map $\cB_E = G \times_{G^{ab}}E^* \rightarrow G \times_{G^{ab}} F^* = \cB_F$ induced by the norm morphism\footnote{Recall that the norm map $N_{E/F}$ may be defined for example as the composite of $\Q$-group maps $E^* \hookrightarrow \Res_{F/\Q} GL_F(E) \map{det} F^*$.} $N_{E/F}: E^* \rightarrow F^*$. Visibly $h_F = N_{E/F} \circ h_E$, so this induces the map of Shimura data claimed.
\end{proof}

\subsection{Abelian type case}
We now return to integral canonical models of standard principal bundles. Fix $(G_2,\fX_2)$ a Shimura datum of abelian type unramified away from $N$, $G_2/\Z[1/N]$ a reductive integral model and $K^N_2=\prod_{p \nmid  N} G_2(\Zp)$. The goal of this section is to prove the following.

\begin{thm}
There exists an integral canonical model for $P_{K^N_2}(G_2,\fX_2)$.
\end{thm}

\subsubsection{}

As in (\ref{abelian model construction}) we may find $(G,\fX)$ of Hodge type with $G$ unramified away from $N$ with $j:G^{der} \rightarrow G_2^{der}$ a central isogeny inducing a morphism of connected Shimura data $(G^{der},\fX^+) \rightarrow (G_2^{der},\fX_2^+)$. By (\ref{abelian isogeny integral}) we may take $G/\Z[1/N]$ reductive such that $j$ is defined over $\Z[1/N]$.

By the construction of the previous section, we may obtain a diagram of Shimura data
$$
\begin{CD}
   (\cB,\fX_\cB)  @>>> (G,\fX) \\
@VVV        @.\\
    (\cB_2,\fX_{\cB,2}) @>>>  (G_2,\fX_2).
\end{CD}
$$

The plan has three stages. We first combine our constructions in the special and Hodge type cases to construct a canonical integral model of $P(\cB,\fX_{\cB})$. We next pass to connected components and take a quotient to transfer this bundle from a $\cB^c$-bundle on $\Sh(\cB,\fX_{\cB})^+$ to a $G_2^c$-bundle on $\Sh(G_2,\fX_2)^+$ defined over the field $E_N$ over which the connected components split off. We finally then assemble over the whole Shimura variety and descend these models making liberal use of the uniqueness statement for canonical models in place of the extension property from \S2.5. We believe this framework could be useful in other contexts where one wants to extend a ``$G$-valued'' construction over Hodge type Shimura varieties to abelian type Shimura varieties.

\subsubsection{}
We start out by checking this picture respects the integral structures on $G,G_2$. Since $E:= E(G,\fX)$ is absolutely unramified at $p\nmid N$ and $E_2 \subset E$ (hence $E_2^*,E^*$ are smooth tori over $\Z[1/N]$) we have a diagram defined over $\Z[1/N]$
$$G \leftarrow \cB=G \times_{G^{ab}} E^* \rightarrow \cB_2 = G_2 \times_{G_2^{ab}} E_2^* \rightarrow G_2.$$
We remark that the middle map is a composite of (\ref{distconst functorial}) (1) and (2), the first of which involves a change of derived group (and by construction $G^{der} \rightarrow G_2^{der}$ is defined over $\Z[1/N]$), and the second of which is a norm map $N_{E/E_2}:E^* \rightarrow E_2^*$ which clearly respects integral structures. We let $K^N,K_2^N,K_{\cB}^N$ denote the obvious corresponding hyperspecial prime to $N$ level structures and note there is a corresponding diagram of integral models of Shimura varieties
$$\cS_{K^N}(G,\fX) \leftarrow \cS_{K_{\cB}^N}(\cB,\fX_{\cB}) \rightarrow \cS_{K_2^N}(G_2,\fX_2).$$

Let us first, now this diagram is in place, check the following condition (\ref{all the points}), which we recall is important in guaranteeing the uniqueness of canonical models.

\begin{prop}
\label{all the abelian}
With notation as above, $\cS_{K_\cB^N}(\cB,\fX_{\cB})$ and $\cS_{K_2^N}(G_2,\fX_2)$ both have all the crystalline points.
\end{prop}
\begin{proof}
Let $x \in \cS_{K_\cB^N}(\cB,\fX_{\cB})(W(k))$, $K_0 = W(k)[1/p]$. Also let $W(k)^{ur}$ be the ring of integers of the maximal unramified (algebraic) extension $K_0^{ur}/K_0$. Note that its images in $\cS_{K^N}(G,\fX)$ and $\cS_{\hat{E^{*,N}}}(E^*,h_E)$
 are both crystalline (since these Shimura data have all the crystalline points). To check $x$ is crystalline, it will suffice to check $\omega_{et,x}(V)$ is crystalline for $V$ a faithful representation of $\cB$. Let us take $G \hookrightarrow GL(V_1)$ and $E^* \hookrightarrow GL(V_2)$ and consider the representation
$$\cB \subset G \times E^* \hookrightarrow GL(V_1) \times GL(V_2) \subset GL(V_1 \oplus V_2).$$
Considering the diagram of Shimura varieties and towers at infinite level at $p$, it is clear that $\Gamma_{K_0} \rightarrow GL(\omega_{et}(V_1 \oplus V_2))$ acts via $\Gamma_{K_0} \rightarrow GL(\omega_{et}(V_1)) \times GL(\omega_{et}(V_2))$, each projection of which is crystalline, and this is enough to give the first part of our proposition.

We turn our attention to $\cS_{K_2^N}(G_2,\fX_2)$. First, since it receives a map from the Shimura variety $\cS_{K_\cB^N}(\cB,\fX_{\cB})$ and we observe that the image of a crystalline point is crystalline, and moreover at finite level these maps are finite \'etale, we deduce that in particular all $W(k)^{ur}$ points on the geometric connected component $\cS_{K_2^N}^+(G_2,\fX_2)_{W(k)^{ur}}$ are crystalline.

To show that any point $x \in \cS_{K_2^NK_{2,N}}(W(k))$ is crystalline, recall that Kisin's integral model at level $K_{2,p} := G_2(\Zp)$ over $W(k)$, $\cS_{K_{2,p}}(G_2,\fX_2)_{W(k)}$ has a $G(\Afp)$-action that acts transitively on geometric connected components and acts equivariantly with regard to the tower at $p$ on the generic fibre. Therefore, taking a lift $\tilde{x}$ of $x$ to level $K_{2,p}$, which we may assume is a $W(k)^{ur}$ point since $\cS_{K_{2,p},W(k)} \rightarrow \cS_{K_{2,p}K^p,W(k)}$ is formally \'etale, and taking a translate $\tilde{x}.a$ by a Hecke operator $a \in G(\Afp)$ such that $\tilde{x}.a$ lies in the connected component, we may deduce that $x$ is crystalline from the fact that all points on the connected component are, and that the crystalline property can be checked after restricting a $\Gamma_{K_0}$ representation to $\Gamma_{K_0^{ur}}$.
\end{proof}
We now proceed with constructing the canonical models. The content of the next step can be extracted as a lemma.
\begin{lem}
\label{torsor fiber product}
Let $S$ be a scheme, and $G \twoheadrightarrow \Delta \leftarrow H$ a diagram of $S$-groups, $\cB = G \times_{\Delta} H$ the fibre product group. Let $X$ be an $S$-scheme and suppose we are given $P_G$ a $G$-torsor and $P_H$ an $H$-torsor on $X$ with the property that there is an isomorphism of $\Delta$-torsors $\theta: P_H \times^H \Delta \rightiso P_G \times^G \Delta.$

Then there is a unique (up to canonical isomorphism) $\cB$-torsor $P_{\cB}$ together with torsor isomorphisms $\theta_G: P_{\cB} \times^\cB G \rightiso P_G$ and $\theta_H: P_{\cB} \times^\cB H \rightiso P_H$ such that the induced isomorphism
$$\theta_H \circ \theta_G^{-1}: P_G \times^G \Delta \rightiso P_{\cB} \times^\cB \Delta \rightiso P_H \times^H \Delta$$
is equal to $\theta$, and it is given by\footnote{We use the notation $X\times_{\theta} Y$ in a situation where we are given maps $X \rightarrow Z_X, Y \rightarrow Z_Y$ and an isomorphism $\theta:Z_X \rightiso Z_Y$ between two schemes to mean the limit of the diagram
$$
\begin{CD}
     @. @.Y  \\
@.  @.      @VVV\\
    X @>>> Z_X @>\theta>\cong> Z_Y.
\end{CD}
$$}

$$P_{\cB} = P_G \times_{\theta} P_H$$
with $\theta_G,\theta_H$ the natural projections.
\end{lem}
\begin{proof}
Let us first show that $P_{\cB} = P_G \times_{\theta} P_H$ is indeed a $\cB$-torsor. We define the $\cB=G \times_{\Delta} H$-action $$(p_G,p_H).(g,h) =: (p_G .g, p_H .h)$$
which is visibly an action. Passing to an \'etale cover $X' \rightarrow X$ over which $P_G$ and $P_H$ admit sections $p_G,p_H$. Since $G$ surjects onto $\Delta$ these may be chosen (perhaps at the cost of passing to a finer \'etale cover) such that if $\pi_G:P_G \rightarrow P_G \times^G \Delta$ and $\pi_H:P_H \rightarrow P_H \times^H \Delta$ are the projection maps, we have $\theta(\pi_G(p_G)) = \pi_H(p_H)$, giving a section $(p_G,p_H) \in P_{\cB}(X')$, from which we may see immediately that it is a $\cB$-torsor. It is also obvious that it has the required property.

Uniqueness is essentially formal, but we give the argument. Suppose we are given $P'_{\cB}$ together with $\theta'_G:P'_{\cB} \times^\cB G \rightiso P_G$ and $\theta'_H:P'_{\cB} \times^\cB H \rightiso P_H$ satisfying the compatibility given. Then by the universal property of $P_{\cB} = P_G \times_{\theta} P_H$ there exists a unique map $\alpha: P'_{\cB} \rightarrow P_{\cB}$ making everything commute. It is easy to check $\alpha$ is a map of $\cB$-torsors, whence it is automatically an isomorphism.
\end{proof}

\subsubsection{}
Let us apply this in the context of our standard principal bundles. If $T$ is a torus unramified away from $N$ we introduce the notation $\hat{T}^N = \prod_{p \nmid N} T(\Zp)$.

Consider the diagram (whose arrows we have named for convenience)
$$
\begin{CD}
    \cS_{K_\cB^N}(\cB,\fX_{\cB}) @>>\delta'> \cS_{\hat{E}^{*,N}}(E^*,h_E) \\
@V\pi VV        @VVN\mu V\\
    \cS_{K^N}(G,\fX) @>>\delta> \cS_{\hat{G}^{ab,N}}(G^{ab},h^{ab}). 
\end{CD}
$$

By (\ref{main hodge}) we have $\cP_{K^N}$ over $\cS_{K^N}(G,\fX)$ an integral canonical model for $P_{K^N}(G,\fX)$, and by (\ref{torus canmod}) we have $\cP_{\hat{E}^{*,N}}, \cP_{\hat{G}^{ab,N}}$ canonical models over the right hand side of the diagram also. By (\ref{func}) there are natural isomorphisms of $G^{ab}$ torsors
$$N\mu^*\cP_{\hat{G}^{ab,N}} \cong \cP_{\hat{E}^{*,N}} \times^{E^{*c}} G^{ab}, \ \ \delta^*\cP_{\hat{G}^{ab,N}} \cong \cP_{K^N} \times^{G} G^{ab}.$$
Pulling these back further and using that the diagram commutes and so we have canonical identifications $\pi^*\delta^* \cong (\delta \pi)^* = (N\mu\delta')^* \cong \delta'^*(N\mu)^*$, we obtain a natural isomorphism
$$\theta: \pi^*(\cP_{K^N} \times^G G^{ab}) \rightiso \delta'^*(\cP_{\hat{E}^{*,N}} \times^{E^{*c}} G^{ab}).$$

Note further that any integral canonical model $\cP_{K_\cB^N}$ for $P_{K_\cB^N}(\cB,\fX_\cB)$ will by (\ref{func}) be required to satisfy the conditions of (\ref{torsor fiber product}) with respect to this isomorphism $\theta$. Therefore by (\ref{torsor fiber product}) if it exists it is given by
$$\cP_{K^N_\cB} := \pi^*\cP_{K^N} \times_{\theta} \delta'^*\cP_{\hat{E}^{*,N}},$$
together with the identification
\begin{eqnarray*}
\iota_\cB: \cP_{K^N_\cB} \otimes_{\cO_E[1/N]} E & = & (\pi^*\cP_{K^N} \times_{\theta} \delta'^*\cP_{\hat{E}^{*,N}})\otimes_{\cO_E[1/N]} E \\
& \map{(\iota_G,\iota_{E^*})} & \pi^*P_{K^N}(G,\fX) \times_{\theta \otimes E} \delta'^*P_{\hat{E}^{*,N}}(E^*,h_E)\\
& = & P_{K^N_\cB}(\cB,\fX_\cB).\\
\end{eqnarray*}

\begin{prop}
As defined above, $(\cP_{K^N_{\cB}}, \iota_\cB)$ is an integral canonical model for $P_{K^N_\cB}(\cB,\fX_\cB)$ with equivariant $\cA_P^N(\cB)$-action.\footnote{Recall we defined this group in (\ref{ad extension}).}
\end{prop}
\begin{proof}
Our task is to show the connection, filtration and $\cA_P^N(\cB)$-actions extend, and check the lattice property.

We first study the connection. Let $p_i, i=1,2$ be the projections from a first order neighbourhood of the diagonal, abusively retaining the same notation for morphisms of schemes and their first order thickenings. Then\footnote{We also abuse notation here and confuse $\pi,\delta'$ with their induced maps on first order neighbourhoods of the diagonal.}
$$p_i^*\cP_{K_{\cB}^N} = p_i^*(\pi^*\cP_{K^N} \times_{\theta} \delta'^*\cP_{\hat{E}^{*,N}}) \cong \pi^*p_i^*\cP_{K^N} \times_{p_i^*(\theta)} \delta'^* p_i^*\cP_{\hat{E}^{*,N}},$$
and in these co-ordinates, $\nabla_{\cB} = (\pi^*\nabla_G,\delta'^*\id)$ is the connection required so we see directly that it extends.

For the filtrations, it is easy to check using \cite[5.3.4]{con} that in fact we may naturally identify $$\mathcal{GR}_\mu = G/P_\mu = G^{der}/(P_\mu \cap G^{der}) = \cB/\cP_{\mu_\cB} = \mathcal{GR}_{\mu_\cB}$$
with $G$ and $\cB$-actions factoring through $G^{ad}$, so actually the composite
$$\gamma_{\cB}:\cP_{K_{\cB}^N} \rightarrow \cP_{K^N} \map{\gamma} \mathcal{GR}_\mu = \mathcal{GR}_{\mu_\cB}$$
does the job.

We also note that we already have a natural extension of the Hecke action via
$$\prod_{p|N} \cB(\Qp) \subset \left(\prod_{p|N} G(\Qp)\right) \times \left( \prod_{p|N} E^*(\Qp) \right)$$
which acts via its two projections on $\pi^*\cP_{K^N} \times_{\theta} \delta'^*\cP_{\hat{E}^{*,N}}$.

For the lattice property, we borrow the trick from the first part of the proof of (\ref{all the abelian}), taking a faithful representation $V = V_1 \oplus V_2$ of $\cB$, formed as a sum of faithful representations of $G$ and $E^*$. It is easy to check from the construction that 
$$\omega_{\cP_{K_\cB^N}}(V) = \omega_{\pi^*\cP_{K^N}}(V_1) \oplus \omega_{\delta'^*\cP_{\hat{E}^{*,N}}}(V_2)$$ whence the lattice property follows immediately from that already known for the two terms on the right hand side together with the fact that the image of a crystalline point is always crystalline. Now we know the lattice property, we get a uniqueness statement, and the $\cA_P^N(\cB)$-action extends formally.
\end{proof}

\subsubsection{}
We now pass to a connected component 
$$\Sh^+_{K_\cB^N}(\cB,\fX_\cB) = \Sh^+_{K^N}(G,\fX),$$
recalling (\ref{component field}) that it is defined over an extension $E_N/E$ unramified away from $N$ and letting $\cO := \cO_{E_N}[1/N]$. We may therefore extend the above identification to 
$$\cS^+_{K_\cB^N}(\cB,\fX_\cB)_{\cO} = \cS^+_{K^N}(G,\fX)_{\cO},$$
 and restrict $\cP_{K^N_\cB}$ to get the $\cB^c$-torsor
$$\cP_{\cB}^+: = \cP_{K^N_{\cB}} \times_{\cS_{K^N_{\cB}}(\cB,\fX_{\cB})_{\cO}} \cS^+_{K^N}(G,\fX)_{\cO} \rightarrow \cS^+_{K^N}(G,\fX)_{\cO}.$$

Recall the group 
$$\Delta^N(G,G_2) = \Ker(\cA^{N,\circ}(G^{der}) \rightarrow \cA^{N,\circ}(G_2^{der}))$$
with the property that 
$$\cS_{K^N}^+(G,\fX)_{\cO}/\Delta^N(G,G_2) = \cS_{K_2^N}^+(G_2,\fX_2)_{\cO}.$$

Let us also recall the extension of group schemes over $\Z[1/N]$
$$1 \rightarrow \cB^c \rightarrow \cA_P^N(\cB) \rightarrow \cA^N(\cB) \rightarrow 1.$$
Forming the pullback along $\cA^{N,\circ}(G^{der}) =\cA^{N,\circ}(\cB) \subset \cA^N(\cB)$ we get an extension
$$1 \rightarrow \cB^c \rightarrow \cA_P^{N,\circ}(\cB) \rightarrow \cA^{N,\circ}(G^{der}) \rightarrow 1$$
which acts on $\cP_{\cB}^+$ equivariantly in the obvious fashion.

\begin{lem}
These group schemes have the following additional properties.
\begin{enumerate}
\item The group $\Xi:=\Ker(\cB^c \rightarrow G_2^c) \subset \cB^c$ is a normal subgroup of $\cA_P^{N,\circ}(\cB)$.
\item The kernel $\Delta_P^{N}(\cB,G_2) := \Ker(\cA^{N,\circ}_P(\cB) \rightarrow \cA^{N,\circ}_P(G_2))$ is canonically an extension
$$1 \rightarrow \Xi \rightarrow \Delta_P^{N}(\cB,G_2) \rightarrow \Delta^N(G,G_2) \rightarrow 1$$
which acts freely on $\cP_{\cB}^+$.
\end{enumerate}
\end{lem}
\begin{proof}
We first remark that (1) is obvious because $\Xi \subset Z(\cB^c)$ which commutes with the whole of $\cA_P^{N,\circ}(\cB)$.

For (2), that $\Xi$ is a subgroup is clear and normality follows by (1). Existence of the maps and exactness in the middle follow in the usual way. That the final map is a surjection follows because for any $(g,\gamma^{-1}) \in \Delta^N(G,G_2)$, it is hit by $(1,g,\gamma^{-1}) \in \Delta_P^N(\cB,G)$. Finally the action on $\cP^+_\cB \rightarrow \cS_{K^N}^+(G,\fX)_\cO$ is free because $\Xi$ acts freely on each fibre, and $\Delta^N(G,G_2)$ acts freely on $\cS^+_{K^N}(G,\fX)_\cO$ by (\ref{delta free}).
\end{proof}

\subsubsection{}
Equipped with this lemma we may construct a $G_2^c$-bundle
$$\cP_2^+ := (\cP_{\cB}^+/\Delta_P^N(\cB,G_2)) \times^{\cB^c/\Xi} G_2^c \rightarrow \cS_{K_2^N}^+(G_2,\fX_2)_{\cO}$$
with equivariant $\cA_P^{N,\circ}(G_2)$-action. Moreover we have
\begin{eqnarray*}
\iota_2: \cP_2^+ \otimes_\cO E_N & = & (\cP_{\cB}^+ \otimes_{\cO} E_N /\Delta_P^N(\cB,G_2)) \times^{\cB^c/\Xi} G_2^c\\ 
& \map{\iota_{\cB}} &(P^+_{K_\cB^N,E_N}/\Delta_P^N(\cB,G_2)) \times^{\cB^c/\Xi} G_2^c = P^+_{K_2^N,E_N},
\end{eqnarray*}
where the final equality follows by working over $\C$ and the argument of \cite[7.2]{milne1}.

\begin{prop}
\label{almostlastprop}
The pair $(\cP_2^+,\iota_2)$ is an integral canonical model for $P_{K_2^N, E_N}^+(G_2,\fX_2)$.
\end{prop}
\begin{proof}
The Hecke action extends by construction. Let us check the connection extends. We know (from direct observation over $\C$) that the $\Delta_P^N(\cB,G_2)$ action is horizontal, so the flat connection $\nabla: p_1^*\cP_\cB^+ \rightiso p_2^*\cP_\cB^+$ descends (and then can obviously be pushed out along $\cB^c/\Xi \rightarrow G_2^c$) to a flat connection $$\nabla: p_1^*\cP_2^+ \rightiso p_2^*\cP_2^+.$$
Since a similar relationship also relates the connections on the generic fibre we have shown that this $\nabla$ extends over $\cO$ the connection on $P_{K_2^N,E_N}$.

To check the lattice condition at a crystalline point $x$, note the commutative diagram of $E_N$-schemes, letting $K_{\cB}^{Np},K_2^{Np}$ be the obvious full level structures at all primes except $N$ and $p$
$$
\begin{CD}
    \Sh_{K_\cB^{Np}}(\cB,\fX_{\cB})|_{\Sh^+_{K_\cB^N}(\cB,\fX_{\cB})} @>>>  \Sh_{K_2^{Np}}(G_2,\fX_{2})|_{\Sh^+_{K_2^N}(G_2,\fX_2)}\\
@VVV        @VVV\\
    \Sh^+_{K_\cB^N}(\cB,\fX_{\cB}) @>>>  \Sh^+_{K_2^N}(G_2,\fX_2).
\end{CD}
$$
Taking a lift $\tilde{x}$ of $x$ to $\cS^+_{K_\cB^N}(\cB,\fX_{\cB})_{\cO}$, which is crystalline since the map $\cS^+_{K_\cB^N}(\cB,\fX_{\cB})_{\cO} \rightarrow  \cS^+_{K_2^N}(G_2,\fX_2)_{\cO}$ is finite \'etale and its source has all the crystalline points, the diagram gives an identification
$$\omega_{et,\tilde{x}} \circ \Res_{\cB^c(\Zp)}^{G_2^c(\Zp)} = \omega_{et,x}: \Rep_{\Zp}(G_2^c) \rightarrow \Rep_{\Zp}(\Gamma_{\Qp^{ur}}).$$
But now the lattice condition is immediate, since equality of lattices can be checked after passing to a finite \'etale cover, and $\cP_{2,x}^+$ pulls back to $\cP_{\cB,\tilde{x}}^+ \times^{\cB^c} G_2^c$ which as a canonical model already has the required property that 
$\omega_{\cP_{\cB,\tilde{x}}^+} = \D \circ \omega_{et,\tilde{x}}$ by its own lattice condition.

It remains to check the filtration $\gamma_B:\cP^+_\cB \rightarrow \mathcal{GR}_{\mu_\cB}$ descends to $\gamma_2: \cP^+_2 \rightarrow \mathcal{GR}_{\mu_2}$, which amounts to showing that the composite
$$\cP^+_\cB \rightarrow \mathcal{GR}_{\mu_\cB} \rightarrow \mathcal{GR}_{\mu_2}$$
is $\Delta_P^N(\cB,G_2)$-invariant. But this is clear: since $\Xi = \Ker(\cB^c \rightarrow G_2^c) \subset Z(\cB^c)$ it acts trivially on $\mathcal{GR}_{\mu_2}$ and $\Delta^N(G,G_2)$ acts trivially because Hecke operators always act trivially on any $Gr_{\mu}$.
\end{proof}

\subsubsection{}
With this in hand, we are finally able to construct an integral canonical model $\cP_2$ for $P_{K_2^N}(G_2,\fX_2)$ in the spirit of (\ref{assemble extension prop}), but with the uniqueness of canonical models used in place of the extension property, because it may be used to canonically extend any isomorphism.

Indeed, let us decompose $\cS_{K_2^N, \cO} = \coprod_{c \in \pi_0} \cS_c$ into components each of which is geometrically integral. By the argument of (\ref{assemble extension prop} (1)) we can find for each $c \in \pi_0$ some $\cG_c/\Z[1/N]$ a model for $G_2$ such that we may identify $\cS_c \cong \cS_{\prod_{p \nmid  N}\cG_c(\Zp), \cO}^+$ via a Hecke operator. For each $c$ we may make such a choice and invoking (\ref{almostlastprop}) for $\cG=\cG_c$ and using the Hecke equivariance of $P_{K^N}$ we obtain a canonical model $\cP_c^+$ for $P_{K^N_2,E_N}(G_2,\fX_2)|_{\cS_c \otimes_{\cO} E_N}.$ Taking the disjoint union of these we obtain
$$\cP_{2,\cO} := \coprod_{c \in \pi_0} \cP_c^+,$$
an integral canonical model minus a full Hecke action. But by uniqueness (\ref{uniqueness}) of canonical models, it is formal to extend the Hecke operators acting between components. Finally by (\ref{galois descent}) it descends to $\cP_2/\cO_{E_2}[1/N]$, an integral canonical model for $P_{K_2^N}(G_2,\fX_2)/E_2$. Thus our main theorem has been proved.

\subsection{Automorphic vector bundles}
With our integral canonical models $\cP_{K^N}(G,\fX)$ constructed in the abelian type case, we should discuss the construction of automorphic vector bundles (recall the discussion (\ref{avbs})) in this setting, although it requires no new ideas.

\begin{thm}
Let $(G,\fX)$ be a Shimura datum of abelian type, $\cG/\Z[1/N]$ a reductive model for $G$, $K^N = \prod_{p \nmid N} \cG(\Zp)$ and for $\mu$ a Hodge cocharacter of $(G,\fX)$, and $L/E=E(G,\fX)$ a finite extension. 

Then we have a canonical functor $\cJ \mapsto \cV(\cJ)$ from $\cG_{\cO_L[1/N]}$-equivariant vector bundles on $\mathcal{GR}_{\mu, \cO_L[1/N]}$ to vector bundles on $\cS_{K^N}$ which on the generic fibre is identified naturally with that of \cite[III 5.1]{milne3}.
\end{thm}

\begin{proof}
Recall that we have the picture
$$\cS_{K^N} \leftarrow \cP_{K^N} \map{\gamma} \mathcal{GR}_\mu.$$
The functor $\cJ \mapsto \cV(\cJ)$ is given by pulling back $\cJ \mapsto \gamma^*\cJ$ and then using the usual equivalence between $\cG$-equivariant vector bundles on $\cP_{K^N}$ and vector bundles on $\cS_{K^N}$. This is obviously a functor, and the compatibility with the usual construction \cite[III 5.1]{milne3} follows because $\cP_{K^N}$ comes with a canonical identification $$\iota: \cP_{K^N} \otimes E \rightiso P_{K^N}(G,\fX)$$ under which $\gamma$ is an extension of the $G$-filtration on the RHS.
\end{proof}

It is also easy to deduce from our construction the following additional functoriality property.
\begin{prop}
We are given a morphism $f: (G_1,\fX_1,\cG_1) \rightarrow (G_2,\fX_2,\cG_2)$ of Shimura data of abelian type together with reductive models for $G_i$ over $\Z[1/N]$, $L/E(G_1,\fX_1)$ finite, and $\mu_1,\mu_2$ Hodge cocharacters for $\fX_1,\fX_2$ respectively. Suppose we are also given $\cJ_2$ a $\cG_2$-equivariant vector bundle on $\mathcal{GR}_{\mu_2,\cO_L[1/N]}$. Pulling back $\cJ_2$ along $$\mathcal{GR}_{\mu_1, \cO_L[1/N]} \rightarrow \mathcal{GR}_{\mu_2,\cO_L[1/N]}$$ and restricting the $\cG_2$ action to $\cG_1$ we obtain a $\cG_1$-equivariant vector bundle $\cJ_1$.

Then there is a canonical identification of vector bundles over $\cS_{\prod_{p \nmid  N}\cG_1(\Zp)}(G_1,\fX_1)$
$$\cV(\cJ_1) \cong f^*\cV(\cJ_2).$$
\end{prop}
\begin{proof}
This is immediate from the construction of $\cV(-)$ and (\ref{func}).
\end{proof}


\begin{thebibliography}{10}
\bibitem{bms}
   Bhatt, B., Morrow, M., Scholze, P.
   \emph{Integral $p$-adic Hodge Theory}, preprint, 2016.

\bibitem{blasius}
  Blasius, D.
  \emph{A $p$-adic property of Hodge classes on abelian varieties}, Motives (Seattle, WA, 1991), Proc. Sympos. Pure Math., vol. 55, Amer. Math. Soc., Providence, RI, 1994, pp. 293-308.

\bibitem{brosh}
  Broshi, M.
  \emph{$G$-Torsors over a Dedekind scheme}, J. of Pure and App. Alg. 217 (2013), 11-19.

\bibitem{con}
  Conrad, B.
  \emph{Reductive Group Schemes}, Lecture notes of the summer school ``Group schemes'' in Luminy, 2011.

\bibitem{del1}
  Deligne, P.
  \emph{Travaux de Shimura}, Exp 389, S\'eminaire Bourbaki (1970/1971), Lecture notes in Math. 244, Springer, pp. 123-165, 1971.

\bibitem{del2}
  Deligne, P.
  \emph{Vari\'et\'es de Shimura: interpr\'etation modulaire, et techniques de construction de mod\`eles canoniques}, Automorphic forms, representations and L-functions (Corvallis 1977), Proc. Sympos. Pure Math XXXIII, Amer. Math. Soc, pp. 247-289, 1979.

\bibitem{sga3}
  Demazure, M., Grothendieck, A.
  \emph{Sch\'emas en groupes I, II, III}, Lecture Notes in Math. 151-153, Springer, 1962-1970.

\bibitem{har1}
  Harris, M.
  \emph{Arithmetic vector bundles and automorphic forms on Shimura varieties, I}, Invent. math. 82, 1985, pp. 151 - 189.

\bibitem{ip}
  Ichino, A., Prasanna, K.
  \emph{Periods of quaternionic Shimura varieties, I,} preprint, 2015.

\bibitem{kim-mp}
  Kim, W., Madapusi Pera, K.
  \emph{2-adic integral canonical models and the Tate conjecture in characteristic 2}, Forum Math. Sigma 4 (2016), e28.

\bibitem{kis1}
  Kisin, M.
  \emph{Crystalline representations and F-crystals}, Algebraic geometry and number theory, Progr. Math 253, Birkh\"auser Boston, pp. 459-496, 2006.
 
 \bibitem{kis2}
  Kisin, M.
  \emph{Integral models for Shimura varieties of abelian type}, J. AMS 23(4) (2010), 967-1012.

\bibitem{kis3}
  Kisin, M.
  \emph{Mod $p$ points on Shimura varieties of abelian type}, preprint, 2013.

\bibitem{kispap}
  Kisin, M., G. Pappas 
  \emph{Integral models of Shimura varieties with parahoric level structure}, preprint.

\bibitem{l}
  Lovering, T.
  \emph{Filtered $F$-crystals on Shimura varieties of abelian type}, preprint 2017.

\bibitem{mp}
  Madapusi Pera, K.
  \emph{Toroidal compactifications of integral models of Shimura varieties of Hodge type}, preprint 2015.

\bibitem{maul}
  Maulik, D.
  \emph{Supersingular K3 surfaces for large primes}, Duke Math. J. 163.13 (2014), pp. 2357-2425.

\bibitem{milne1}
  Milne, J.S.
  \emph{Automorphic vector bundles on connected Shimura varieties}, Invent. math. 92 (1988), 91-128.
  
\bibitem{milne2}
  Milne, J.S.
  \emph{The points on a Shimura variety modulo a prime of good reduction}, in Langlands, R.P. and Ramakrishnan, D. (eds.): The Zeta Functions of Picard Modular Surfaces, Les publications CRM, Montreal (1992) 151-253.
  
\bibitem{milne3}
  Milne, J.S.
  \emph{Canonical Models of (Mixed) Shimura Varieties and Automorphic Vector Bundles}, in Automorphic Forms, Shimura Varieties, and L-functions, Perspectives in Math. 10, 1990, pp. 283-414.
  
\bibitem{milne9}
  Milne, J.S.
  \emph{Shimura varieties and motives}, pp. 447-523. In Motives (Seattle, WA, 1991), Proc. Sympos. Pure Math. Amer. Math. Soc., Providence, RI.

\bibitem{moon}
  Moonen, B.
  \emph{Models of Shimura varieties in mixed characteristics}, Galois representations in arithmetic algebraic geometry (Durham, 1996), London Math. Soc. Lecture Note Ser. 254, Cambridge Univ. Press, pp. 267-350, 1998.


\bibitem{mum}
  Mumford, D., Fogarty, J., Kirwan, F.
  \emph{Geometric Invariant Theory}, Springer, Heidelberg, 1965.
  
\bibitem{wat}
 Waterhouse, W.
 \emph{Introduction to Affine Group Schemes}, Springer, Heidelberg, 1979.
  
\end{thebibliography}
\end{document}